\documentclass{amsart}

\usepackage{verbatim}
\usepackage{amssymb}
\usepackage{upref}
\usepackage[all]{xy}
\usepackage{color}
\usepackage{paralist}

\allowdisplaybreaks

\newtheorem{thm}{Theorem}[section]
\newtheorem{lem}[thm]{Lemma}
\newtheorem{cor}[thm]{Corollary}
\newtheorem{prop}[thm]{Proposition}

\theoremstyle{definition}
\newtheorem{defn}[thm]{Definition}

\theoremstyle{remark}  
\newtheorem{rem}[thm]{Remark}
\newtheorem*{rem*}{Remark}
\newtheorem{step}{Step}
\newtheorem*{warn}{Warning}

\numberwithin{equation}{section}

\newcommand{\thmref}[1]{Theorem~\textup{\ref{#1}}}
\newcommand{\corref}[1]{Corollary~\textup{\ref{#1}}}
\newcommand{\lemref}[1]{Lemma~\textup{\ref{#1}}}
\newcommand{\propref}[1]{Proposition~\textup{\ref{#1}}}

\newcommand{\midtext}[1]{\quad\text{#1}\quad}
\newcommand{\righttext}[1]{\quad\text{#1 }}
\renewcommand{\and}{\midtext{and}}

\newcommand{\all}{\righttext{for all}}

\renewcommand{\(}{\textup(}
\renewcommand{\)}{\textup)}

\usepackage{mathrsfs}
\newcommand\bbfont{\mathbf}
\newcommand\groupoidfont{\mathcal}
\newcommand\bundlefont{\mathscr}
\newcommand{\C}{\bbfont C}

\newcommand{\KK}{\groupoidfont K}

\newcommand{\EE}{\bundlefont E}
\newcommand{\BB}{\bundlefont B}
\newcommand{\XX}{\groupoidfont X}
\newcommand{\YY}{\groupoidfont Y}
\renewcommand{\AA}{\bundlefont A}

\DeclareMathOperator{\aut}{Aut}

\DeclareMathOperator{\ad}{Ad}
\DeclareMathOperator{\ind}{Ind}

\newcommand{\id}{\text{\textup{id}}}

\newcommand{\<}{\langle}
\renewcommand{\>}{\rangle}

\newcommand{\under}{\backslash}
\newcommand{\inv}{^{-1}}

\renewcommand{\bar}{\overline}

\newcommand{\lt}{\textup{lt}}

\newcommand{\cs}%
{\ensuremath{\mathbf{C^*}}}

\usepackage{color}
\definecolor{Dgreen}{cmyk}{0.93,0.33,0.92,0.25} 

\begin{document}
\title{Fell Bundles and Imprimitivity Theorems}

\author[Kaliszewski]{S. Kaliszewski}
\address{Department of Mathematics and Statistics
\\Arizona State University
\\Tempe, Arizona 85287}
\email{kaliszewski@asu.edu}

\author[Muhly]{Paul S. Muhly}
\address{Department of Mathematics
\\The University of Iowa
\\Iowa City, IA 52242}
\email{paul-muhly@uiowa.edu}

\author[Quigg]{John Quigg}
\address{Department of Mathematics and Statistics
\\Arizona State University
\\Tempe, Arizona 85287}
\email{quigg@asu.edu}

\author[Williams]{Dana P. Williams}
\address{Department of Mathematics
\\Dartmouth College
\\Hanover, NH 03755}
\email{dana.williams@dartmouth.edu}


\date{June 25, 2012}



\begin{abstract}
  Our goal in this paper and two sequels is to apply the
  Yamagami-Muhly-Williams equivalence theorem for Fell bundles over
  groupoids to recover and extend all known imprimitivity theorems involving
  groups.  Here we extend Raeburn's symmetric imprimitivity
  theorem, and also, in an appendix, we develop a number of tools for
  the theory of 
  Fell bundles that have not previously appeared in the literature.
\end{abstract}
\maketitle

\section{Introduction}
\label{intro}

In this paper, we take up in earnest the program that was suggested in
\cite{kmqw1} to use Fell bundles to unify and extend a broad range of
imprimitivity theorems and equivalence theorems for $C^*$-dynamical
systems, especially in settings that involve nonabelian duality. In a
sense, it has long been understood that Fell bundles provide an
important mechanism for illuminating the structure of $C^*$-dynamical
systems and their associated crossed products.  Fell invented Fell
bundles precisely to understand better and extend the theory of
induced representations that had been built up around Mackey's program
\cite{fel:ajm69, fel:mams69}. What is novel about our contribution is
the use of Fell bundles over \emph{groupoids}. Indeed, Fell bundles
over groupoids appear to be essential in important situations
involving groups acting and coacting on $C^*$-dynamical systems. For
example, as we showed in \cite{kmqw1}, if one has a Fell bundle over a
locally compact \emph{group}, $\AA \to G$, then there is a natural
coaction $\delta$ of $G$ on the $C^*$-algebra of the bundle,
$C^*(G,\AA )$. Further, the cocrossed product, $C^*(G,\AA )
\rtimes_{\delta} G$, is naturally isomorphic to a Fell bundle over the
transformation groupoid determined by the action of $G$ on $G$ through
left translation \cite[Theorem 5.1]{kmqw1}. This fact turns out to be
crucial for proving that the natural surjection from $C^*(G,\AA )
\rtimes_{\delta} G\rtimes_{\widehat{\delta}} G$ to $C^*(G,\AA )\otimes
\mathcal{K}(L^2(G))$, is an isomorphism \cite[Theorem 8.1]{kmqw1} ---
in other worlds, that $\delta$ is a maximal coaction. (Here
$\widehat{\delta}$ denotes the natural action of $G$ on $C^*(G,\AA )
\rtimes_{\delta}G$ that is dual to $\delta$, and $\mathcal{K}(L^2(G))$
denotes the compact operators on $L^2(G)$.) This same technique ---
using Fell bundles over groupoids to prove that dual coactions
on full cross-sectional $C^*$-algebras of Fell bundles over groups are 
maximal --- was used in \cite{ech} for discrete groups.

This is the first of three papers that are dedicated to showing how
all known imprimitivity theorems can be unified and extended under the
umbrella of Fell bundles over groupoids. The notion of a system of
imprimitivity, and the first imprimitivity theorem, appeared very
early in the theory of groups. As discussed in \cite[page 64
ff]{Hall1959}, if a group $G$ acts on a set $X$, then a system of
imprimitivity is simply a partition $\mathcal{P}$ of $X$ that is
invariant under the action of $G$, in the sense that for each $P \in
\mathcal{P}$ and for each $x\in G$, $Px$ is another element of
$\mathcal{P}$. The system is called transitive if $G$ permutes the
elements of $\mathcal{P}$ transitively. Suppose, in addition, that
each $P\in \mathcal{P}$ is a vector space and that the action of $G$
is linear in the sense that each $x\in G$ induces a linear map from $P
\in \mathcal{P}$ to $Px$. Then in a natural way, $G$ acts linearly on
the direct sum ${\bigoplus}_{P\in \mathcal{P}}P$, yielding a
representation of $G$, $U = \{U_x\}_{x\in G}$. Further, for each $P
\in \mathcal{P}$, the restriction $V$ of $U$ to the isotropy group $H$
of $P$ is a representation of $H$ on $P$ and $U$ is \emph{induced}
from $V$ \cite[Theorem 16.7.1]{Hall1959}. And conversely, if $H$ is
the isotropy group of any $P \in \mathcal{P}$ and if $V$ is a
representation of $H$ on $P$, then in a natural way $V$ induces a
representation of $G$ on ${\bigoplus}_{P\in \mathcal{P}}P$. These two
statements, taken together, are known as the imprimitivity
theorem. Thus, even in the setting of finite groups, one begins to see
the players that enter into our analysis: the elements of
$\mathcal{P}$ may be viewed as a bundle over $G/H$, where $H$ is the
subgroup of $G$ that fixes a particular $P\in \mathcal{P}$. The
representation space ${\bigoplus}_{P\in \mathcal{P}}P$ is the space of
all cross sections of this bundle. The representation $U$ of $G$ is
induced by an action of $H$. We may also think of ${\bigoplus}_{P\in
  \mathcal{P}}P$ as arising from a coaction of $G$, or better of
$G/H$.

Inspired by problems in quantum mechanics, Mackey discovered a
generalization of the imprimitivity theorem that is valid for
\emph{unitary} representations of locally compact groups in
\cite{mackey}. His theorem may be formulated as follows.  Suppose a
second countable locally compact group $G$ acts transitively and
measurably on a standard Borel space $X$, and suppose $H$ is the
isotropy group of a point in $X$. Then a unitary representation $U$ of
$G$ on a Hilbert space $\mathcal{H}$ is induced from a unitary
representation of $H$ if and only if there is a spectral measure $E$
defined on the Borel sets of $X$ with values in the projections on
$\mathcal{H}$ such that for every Borel set $M$ in $X$ and every $g
\in G$,\begin{equation} \label{covariance} E(Mg) = U_{g}^{-1} E(M) U_g
\end{equation}

The connection with the imprimitivity theorem for finite groups
becomes clear once one uses direct integral theory (applied to the
spectral measure $E$) to decompose $\mathcal{H}$ as the direct
integral of a bundle $\mathfrak{H} = \{H_x\}_{x\in X}$ of Hilbert
spaces over $X$. The representation $U$ permutes the fibres $H_x$
transitively and is induced in the fashion indicated above from a
representation of the isotropy group of any point in $X$.  Of course,
there are many technical difficulties to surmount when one works at
this level of generality, but the idea is clear.

The contemporary view of Mackey's theorem is due to Rieffel
\cite{rie:induced}.  He was motivated not only by Mackey's theorem,
but also by other generalizations of it.  In particular, he received
considerable inspiration from Takesaki's paper \cite{takesaki} and
from the work of Fell \cite{fel:mams69}.

To understand Rieffel's perspective, observe first that if a second
countable locally compact group acts measurably and transitively on a
standard Borel space $X$ and if $H$ is the isotropy group of some
point in $X$, then $X$ is Borel isomorphic to the coset space
$G/H$. Further, specifying a spectral measure on $X$ is tantamount to
specifying a $C^*$-representation, say $\pi$, of $C_{0}(G/H)$. The
spectral measure satisfies the covariance equation \eqref{covariance}
relative to some unitary representation $U$ of $G$ if and only if $U$
and $\pi$, satisfy the following covariance equation
\begin{equation}\label{covariance_bis}
  \pi(f_g)= U^{-1}_g \pi(f) U_g
\end{equation}
for all $f\in C_0(G/H)$ and all $g\in G$.  This equation, in turn,
means that the pair $(\pi, U)$ can be \emph{integrated} to a
$C^*$-representation of the crossed product $C^*$-algebra
$C_0(G/H)\rtimes G$. Further, the space $C_c(G)$ may be endowed with
(pre-) $C^*$-algebra-valued inner products in such a way that the
completion $\mathcal{X}$ becomes a \emph{ $(C_0(G/H)\rtimes G)-C^*(H)$
  imprimitivity bimodule}, also called a \emph{Morita equivalence
  bimodule}, that links $C_0(G/H)\rtimes G$ and $C^*(H)$. The group
$C^*$-algebra, $C^*(G)$, sits inside the multiplier algebra of
$C_0(G/H)\rtimes G$ and one concludes that a representation $\rho$ of
$C^*(G)$ is induced from a representation $\sigma$, say, of $C^*(H)$
via $\mathcal{X}$ if and only if $\rho$ can be extended in a natural
way to a $C^*$-representation of $C_0(G/H)\rtimes G$.

Thus, Rieffel observed that Mackey's imprimitivity theorem is a
special case of the following theorem:

\begin{thm}[{\cite[Theorem 6.29]{rie:induced}}]\label{Rieffel_Induced}
  Suppose $\pi$ is a representation of a $C^*$-algebra $A$ on a
  Hilbert space $H_{\pi}$. Suppose also that there is a representation
  of $A$ in the space, $\mathcal{L}(\mathcal{X}_B)$, of bounded,
  adjointable operators on the right Hilbert $C^*$-module
  $\mathcal{X}_B$ over the $C^*$-algebra $B$. Then $\pi$ is induced
  from a representation $\sigma$ of $B$ on a Hilbert space
  $H_{\sigma}$ in the sense of Rieffel \cite[Definition
  5.2]{rie:induced} if and only if $\pi$ can be extended to a
  representation of the compact operators,
  $\mathcal{K}(\mathcal{X}_B)$, on $\mathcal{X}_B$, in such a way that
  \[
  \pi(ak) = \pi(a)\pi(k)
  \]
  for all $a \in A$ and $k \in \mathcal{K}(\mathcal{X}_B)$.

\end{thm}

Rieffel's theorem opens a whole new dimension to representation
theory, not only of groups but of $C^*$-algebras, generally. One is
led ineluctably to look for situations when a given $C^*$-algebra $A$
may be represented in $\mathcal{L}(\mathcal{X}_B)$ for suitable
$C^*$-algebras $B$ and Hilbert $C^*$-modules $\mathcal{X}_B$. Such
searches are really searches for Morita contexts
$(\mathcal{K}(\mathcal{X}_B),\mathcal{X}_B, B)$ that reflect
properties of $A$. Thus, an imprimitivity theorem arises whenever one
finds an interesting Morita context. This point was made initially by
Rieffel in \cite{rie:induced} and was reinforced in his note
\cite{rie:applications}.  Subsequently, in \cite{gre:local}, Green
parlayed Rieffel's
\emph{Imprimitivity-cum-Morita-equivalence}-perspective into a
recovery of Takesaki's theorem in the form of a Morita equivalence
between $(A\otimes C_0(G/H))\rtimes G$ and $A\rtimes H$, where $A$ is
a $C^*$-algebra on which a locally compact group $G$ acts continuously
via $*$-automorphisms.  In the same paper, Green also proved another
imprimitivity theorem for induced actions. (We will have more to say
about induced actions below.)

With the advent of the important theory of nonabelian crossed product 
duality, involving coactions as well as actions of locally compact groups, 
a dual version of Green's theorem became exigent.
This was
supplied by Mansfield \cite{man}. It took the form of a Morita
equivalence between $A\rtimes_{\delta} G\rtimes_{\widehat{\delta}} N$
and $A\rtimes_{{\delta}|} G/N$ when $N$ is a normal subgroup of
$G$. (Here, $\delta$ is a coaction of the locally compact group $G$ on
the $C^*$-algebra $A$, $\widehat{\delta}$ is the dual action of $G$
determined by $\delta$, and ${\delta}|$ is the natural coaction of
$G/N$.) Mansfield's result was subsequently generalized to non-normal
$N$ by an~Huef and Raeburn \cite{HR:mansfield}. Their study, in turn,
led to the notion of a cocrossed product by a coaction of a
homogeneous space.

A few other imprimitivity theorems involving crossed products by group
actions have appeared: Combes \cite{com} showed that an equivariant
Morita equivalence gives rise to a Morita equivalence between the
crossed products. Raeburn's symmetric imprimitivity theorem
\cite{rae:symmetric} recovers, among other results, Green's theorem
for induced actions, the Green-Takesaki theorem for induced
representations of $C^*$-dynamical systems, and several of the
examples developed in \cite{rie:applications}.

Another general imprimitivity theorem that unifies various results,
starting with Mackey's, was given by Fell (see, for example,
\cite{fd2}), using what are now known as Fell bundles over locally
compact groups.  Introducing groupoids to the realm of imprimitivity
theorems, Muhly, Renault, and Williams \cite{mrw} showed how a certain
type of equivalence between locally compact groupoids gives rise to
Morita equivalence between their $C^*$-algebras.  In
\cite{ren:representation}, Renault combined groupoid equivalence with
groupoid crossed products, generalizing both Raeburn's symmetric
imprimitivity theorem and the Muhly-Renault-Williams groupoid algebra
imprimitivity theorem.

In an unpublished preprint \cite{yam:symmetric}, Yamagami stated a
very general imprimitivity theorem for Fell bundles over locally
compact groupoids (see also \cite{yam:memoir}, \cite{muhly:fell}), and
the complete details (in slightly greater generality) were worked out
by Muhly and Williams in \cite{mw:fell}. This imprimitivity theorem is
central to our considerations, as we indicated at the outset, and so
we will refer to this as the MWY theorem.

It is the MWY theorem that we will use to unify all known
imprimitivity theorems involving groups.  While an outline of what is
necessary is fairly clear, the details are formidable and require a
large amount of work. Due to their length, we will split this project
over several papers.  In this first paper, we will show how the MWY
theorem can be used to deduce Raeburn's symmetric imprimitivity
theorem, and thereby unify many imprimitivity theorems involving
crossed products by actions of locally compact groups.  To prepare the
way for Raeburn's theorem, we first prove a general imprimitivity
theorem, which we call the ``Symmetric Action theorem'', for commuting
group actions by automorphisms on a Fell bundle. The Symmetric Action
theorem will quickly imply not only Raeburn's theorem, but also
Mansfield's theorem (which we postpone to a subsequent paper), thus
giving a unified approach to the standard imprimitivity theorems for
actions \emph{and} for coactions.  (We could use the Symmetric Action
theorem to quickly derive the Green and the Green-Takesaki
imprimitivity theorems, but we leave details aside since Raeburn has
already shown how his symmetric imprimitivity theorem quickly implies
the Green and Green-Takesaki theorems.)  In subsequent papers, in
addition to Mansfield's imprimitivity theorem, we will illuminate a
curious connection between a one-sided version of the Symmetric Action
theorem and Rieffel's imprimitivity theorem for his
generalized-fixed-point algebras. And we shall deduce Fell's original
imprimitivity theorem from the MWY theorem.

To help with some of the technicalities that arise in our
constructions of Fell bundles, we have included in an appendix a
``toolkit for Fell bundle constructions''. In addition to playing an
important role in our present analysis, we believe it will prove
useful elsewhere.  The constructions we develop in the appendix
include semidirect products of Fell bundles (over groupoids) by
actions of locally compact groups, quotients of Fell bundles by free
and proper group actions, and a combination of these two that involve
commuting actions of two groups.  In addition, we give a structure
theorem that characterizes all free and proper actions by
automorphisms of a group on a Fell bundle. For this purpose, we employ
a result, which we believe is due to Palais, that shows that such
actions all arise from transformation Fell bundles (the theory of
which we developed in \cite{kmqw1}).

One final remark before getting down to business.  One may wonder if
all of our results have full groupoid variants.  That is, one may
wonder if we may replace all groups and group actions that we will be
discussing by groupoids and groupoid actions.  That may be possible
and the more optimistic ones among us believe that it is.  However,
the technical details appear to be more formidable than those we must
develop in this paper and its sequels.  We feel, therefore, that it
will be best to put them aside until the theory of group actions,
along the lines we conceive, are more fully exposed.

\section{Preliminaries}
\label{prelim}

We adopt the conventions of \cite{mw:fell, kmqw1} for Fell bundles
over locally compact groupoids. Whenever we refer to a space (in
particular, to a groupoid or a group), we tacitly assume that it is
locally compact, Hausdorff, and second countable. Whenever we refer to
a Banach bundle over a space (in particular, to a Fell bundle over a
groupoid or a group), we assume that it is upper semicontinuous
and separable --- as in
\cite{mw:fell, kmqw1}. We say that a Fell bundle is \emph{separable}
if the base groupoid is second countable and the Banach space of
continuous sections vanishing at infinity is
separable.  All groupoids will be assumed to be equipped with a
left Haar system.

Whenever we say a groupoid $\XX$ acts on the left of a space $X$, we
tacitly assume that the associated fibring map
\begin{equation*}
\rho:X\to \XX^{(0)}
\end{equation*}
is continuous and open, and that the action is continuous in the
appropriate sense.

If $p:\AA\to \XX$ is a Fell bundle over a locally compact groupoid, we
define $s,r:\AA\to \XX^{(0)}$ by
\begin{equation*}
s(a)=s(p(a))\midtext{and}r(a)=r(p(a)).
\end{equation*}
Similarly, if Fell bundles $\AA\to \XX$ and $\BB\to \YY$ act on the
left and right, respectively, of a Banach bundle $q:\EE\to \Omega$,
with respective fibring maps
\begin{equation*}
\xymatrix{ \XX^{(0)} &\Omega\ar[l]_-\rho \ar[r]^-\sigma &\YY^{(0)}, }
\end{equation*}
we define maps
\begin{equation*}
\xymatrix{ \XX^{(0)} &\EE\ar[l]_-\rho \ar[r]^-\sigma &\YY^{(0)} }
\end{equation*}
by
\begin{equation*}
\rho(e)=\rho(q(e))\midtext{and}\sigma(e)=\sigma(q(e)).
\end{equation*}


\section{The Symmetric Action theorem}
\label{symmetric}

In this section we derive from the Yamagami-Muhly-Williams equivalence
theorem the following general imprimitivity theorem that involves
commuting actions of groups on a Fell bundle.  This theorem will be
used to unify most (but not quite all) of the known imprimitivity
theorems we will derive.  For the background on actions of groups on
Fell bundles and the associated groupoid constructions, see
\S\ref{sec:transf-group-bundl}.

\begin{thm}
  \label{symmetric action}
  If locally compact groups $G$ and $H$ act freely and properly on the
  left and right, respectively, of a Fell bundle $p:\AA\to\XX$ over a
  locally compact groupoid, , and if the actions commute, then $\AA$
  becomes an $(\AA/H\rtimes G)-(H\ltimes G\under\AA)$ equivalence in
  the following way:
  \begin{compactenum}
  \item $\AA/H\rtimes G$ acts on the left of $\AA$ by
    \begin{equation*}
      (a\cdot H,t)\cdot b=a(t\cdot b) \quad\text{if}\quad s(a)=r(t\cdot b);
    \end{equation*}

  \item the left inner product is given by
    \begin{equation*} {}_L\<a,b\> =\bigl(a(t\cdot b^*)\cdot H,t\bigr)
      \quad\text{if}\quad G\cdot s(a)=G\cdot s(b),
    \end{equation*}
    where $t$ is the unique element of $G$ such that $ s(a)=t\cdot
    s(b)$;

  \item $H\ltimes G\under\AA$ acts on the right of $\AA$ by
    \begin{equation*}
      a\cdot (h,G\cdot b)=(a\cdot h)b \quad\text{if}\quad s(a\cdot h)=r(b);
    \end{equation*}

  \item the right inner product is given by
    \begin{equation*}
      \<a,b\>_R =\bigl(h,G\cdot (a^*\cdot h)b\bigr)
      \quad\text{if}\quad r(a)\cdot H=r(b)\cdot H,
    \end{equation*}
    where $h$ is the unique element of $H$ such that $ r(a)\cdot
    h=r(b)$.
  \end{compactenum}
\end{thm}

As we explain in \corref{orbit bundle}, by ``free and proper action on
$\AA$'' we mean that the corresponding action on $\XX$ has these
properties.  By \propref{semidirect orbit bundle action} the action of
$G$ on $\AA$ descends to an action on $\AA/H$, and then the
semidirect-product Fell Bundle $\AA/H\rtimes G\to \XX/H\rtimes G$ acts
on the left of the Banach bundle $\AA$.  Since the hypotheses are
symmetric in $G$ and $H$, with the ``sides'' reversed, we immediately
conclude that
\begin{itemize}
\item $H$ acts on the right of the orbit bundle $G\under\AA$ by $
  (G\cdot a)\cdot h=G\cdot (a\cdot h)$;

\item the semidirect-product Fell bundle $H\ltimes G\under\AA$ acts on
  the right of the Banach bundle $\AA$ by $ a\cdot (h,G\cdot
  b)=(a\cdot h)b$ {if} $s(a\cdot h)=r(b)$.
\end{itemize}

We list, for convenient reference, the formulas for the equivalence of
the base groupoids: by \lemref{semidirect orbit action} the action of
$G$ on $\XX$ descends to an action
\begin{equation*}
  t\cdot (x\cdot H)=(t\cdot x)\cdot H
\end{equation*}
on $\XX/H$, and the semidirect-product groupoid $\XX/H\rtimes G$ acts
on the left of the space $\XX$ by
\begin{equation*}
  (x\cdot H,t)\cdot y=x(t\cdot y) \quad\text{if}\quad s(x)=r(t\cdot y).
\end{equation*}
Again, by symmetry the action of $H$ on $\XX$ descends to an action
\begin{equation*}
  (G\cdot x)\cdot h=G\cdot (x\cdot h)
\end{equation*}
on the orbit groupoid $G\under\XX$, and the semidirect-product
groupoid $H\ltimes G\under\XX$ acts on the right of the space $\XX$ by
\begin{equation*}
  x\cdot (h,G\cdot y)=(x\cdot h)y \quad\text{if}\quad s(x\cdot h)=r(y).
\end{equation*}
We begin by observing that these formulas give a groupoid equivalence:

\begin{lem}
  \label{groupoid equivalence}
  If locally compact groups $G$ and $H$ act freely and properly on the
  left and right, respectively, of a locally compact groupoid $\XX$ by
  automorphisms and if the actions commute, then, with the actions
  indicated above, the space $\XX$ becomes a $(\XX/H\rtimes
  G)-(H\ltimes G\under\XX)$ equivalence.
\end{lem}

\begin{rem*}
  This lemma is a straight-forward generalization of the well-known
  equivalence for transformation groupoids when $\XX$ is just a space.
\end{rem*}

\begin{proof}
  By \cite[Definition~2.1]{mrw} we must verify the following:
  \begin{compactenum}
  \item $\XX/H\rtimes G$ acts freely and properly;
  \item $H\ltimes G\under\XX$ acts freely and properly;
  \item the actions of $\XX/H\rtimes G$ and $H\ltimes G\under\XX$
    commute;
  \item the associated fibre map $\rho:\XX\to \XX^{(0)}/H\times \{e\}$
    factors through a homeomorphism of $\XX/(H\ltimes G\under\XX)$
    onto $\XX^{(0)}/H\times \{e\}$;
  \item the associated fibre map $\sigma:\XX\to \{e\}\times
    G\under\XX^{(0)}$ factors through a homeomorphism of
    $(\XX/H\rtimes G)\under\XX$ onto $\{e\}\times G\under\XX^{(0)}$.
  \end{compactenum}
  Because our hypotheses are symmetric in $G$ and $H$, it will suffice
  to verify (i), (iii), and (iv).  Moreover, we already know (i) from
  \corref{orbit bundle action}, and (iii) is clear.

  So, it remains to verify (iv).  Recall that $\rho$ is defined by
  \begin{equation*}
    \rho(x)=(r(x)\cdot H,e).
  \end{equation*}
  Since the range map $r:\XX\to \XX^{(0)}$ is open, so is $\rho$.
  Thus it suffices to show that $\rho$ factors through a bijection of
  $\XX/(H\ltimes G\under\XX)$ onto $\XX^{(0)}/H\times \{e\}$.  Clearly
  $\rho$ is a surjection of $\XX$ onto $\XX^{(0)}/H\times \{e\}$, and
  it is invariant under the right $(H\ltimes G\under\XX)$-action:
  \begin{align*}
    \rho\bigl(x\cdot (h,G\cdot y)\bigr) &=\rho\bigl((x\cdot h)y\bigr)
    =\bigl(r\bigl((x\cdot h)y\bigr)\cdot H,e\bigr)
    \\&=\bigl(r(x\cdot h)\cdot H,e\bigr) =\bigl(r(x)\cdot H,e\bigr)
    \\&=\rho(x).
  \end{align*}
  Finally, suppose $\rho(x)=\rho(y)$. We must show that $y\in x\cdot
  (H\ltimes G\under\XX)$.  We have $r(x)\cdot H=r(y)\cdot H$, so there
  exists $h\in H$ such that
  \begin{equation*}
    r(y)=r(x)\cdot h=r(x\cdot h).
  \end{equation*}
  Put
  \begin{equation*}
    z=(x\cdot h)\inv y.
  \end{equation*}
  Then
  \begin{equation*}
    y=(x\cdot h)z=x\cdot (h,G\cdot z).  \qedhere
  \end{equation*}
\end{proof}

It will be convenient to record a few formulas associated with the
above groupoid equivalence.  Letting
\begin{equation*}
  \XX*_\sigma \XX=\bigl\{(x,y)\in
  \XX\times\XX:\sigma(x)=\sigma(y)\bigr\},
\end{equation*}
we have a pairing
\begin{equation*} {}_L[\cdot,\cdot]:\XX*_\sigma\XX\to \XX/H\rtimes G
\end{equation*}
characterized by the following property: $ {}_L[x,y] $ is the unique
element $(z\cdot H,t)\in \XX/H\rtimes G$ such that
\begin{equation*}
  x={}_L[x,y]\cdot y=(z\cdot H,t)\cdot y=z\cdot (t\cdot y),
\end{equation*}
so we have
\begin{equation}
  \label{left bracket} {}_L[x,y] =\bigl(x(t\cdot y\inv)\cdot
  H,t\bigr),\quad
  \text{where $t\in G$ is unique with $s(x)=t\cdot s(y)$.}
\end{equation}
Similarly, if $(x,y)\in \XX*_\rho\XX$, so that $\rho(x)=\rho(y)$, then
\begin{equation}
  \label{right bracket} [x,y]_R=\bigl(h,G\cdot (x\inv\cdot h)y\bigr),
  \quad
  \text{where $h\in H$ is unique with $r(y)=r(x)\cdot h$,}
\end{equation}
is the unique element of $H\ltimes G\under\XX$ such that $ y=x\cdot
[x,y]_R$.

\begin{rem}
  Note that in the statement of \thmref{symmetric action} the formulas
  (i) and (iii) are expressed in ``cleaned-up'' form. For example, in
  (i) either $a$ or $b$ has been adjusted within its $H$-orbit to
  force $s(a)=r(b)$. In practice, we might not always have the luxury
  of making such adjustments, so we must be careful in computing the
  left action of $\AA/H\rtimes G$ on $\AA$: for $(a\cdot H,t)\in
  \AA/H\rtimes G$ and $b\in\AA$, the left module product $(a\cdot
  H,t)\cdot b$ is defined if and only if
  \begin{equation*}
    s\(a\cdot H,t)=\rho(b),
  \end{equation*}
  equivalently
  \begin{equation*}
    (t\inv\cdot s(a)\cdot H,e)=(r(b)\cdot H,e),
  \end{equation*}
  which reduces to
  \begin{equation*}
    t\inv\cdot s(a)\cdot H=r(b)\cdot H,
  \end{equation*}
  and then we have
  \begin{equation*}
    (a\cdot H,t)\cdot b =(a\cdot h)(t\cdot b),
  \end{equation*}
  where $h\in H$ is unique with $s(a)\cdot h=t\cdot r(b)$.  Similarly
  for the right action of $H\ltimes G\under\AA$ on $\AA$.
\end{rem}

\begin{proof}
[Proof of \thmref{symmetric action}]
We have seen in \lemref{groupoid equivalence} that the space $\XX$ is
an $(\XX/H\rtimes G)-(H\ltimes G\under\XX)$ equivalence.  By
\propref{semidirect orbit bundle action}, the Fell bundle
$\AA/H\rtimes G$ acts on the left of the Banach bundle $\AA$ as
indicated in (i), and we have discussed at the beginning of this
section how the right action of $H\ltimes G\under\AA$ on $\AA$ arises
from symmetry.

We must verify the axioms in \cite[Definition~6.1]{mw:fell}, which has
three main items \mbox{(a)--(c)}, with item (b) itself having four
parts.  To improve readability we break the verification into steps.

\begin{step}
  Item (a) of \cite[Definition~6.1]{mw:fell} is that the actions of
  $\AA/H\rtimes G$ and $H\ltimes G\under\AA$ on $\AA$ commute: let
  $(a\cdot H,t)\in \AA/H\rtimes G$, $b\in\AA$, and $(h,G\cdot c)\in
  H\ltimes G\under\AA$, with
  \begin{equation*}
    s(a)\cdot H=r(t\cdot b)\cdot H \and G\cdot s(b\cdot h)=G\cdot r(c)
  \end{equation*}
  so that both
  \begin{equation*}
    (a\cdot H,t)\cdot b \and b\cdot (G\cdot c,h)
  \end{equation*}
  are defined.  We must show that both
  \begin{equation*}
    \bigl((a\cdot H,t)\cdot b\bigr)\cdot (G\cdot c,h) \and (a\cdot
    H,t)\cdot \bigl(b\cdot (G\cdot c,h)\bigr)
  \end{equation*}
  are defined, and coincide in $\AA$.

  Adjust $a$ within the orbit $a\cdot H$ so that $s(a)=r(t\cdot
  b)$. Then
  \begin{equation*}
    (a\cdot H,t)\cdot b=a(t\cdot b).
  \end{equation*}
  Similarly, adjust $c$ within $G\cdot c$ so that $s(b\cdot h)=r(c)$,
  and then
  \begin{equation*}
    b\cdot (h,G\cdot c)=(b\cdot h)c.
  \end{equation*}
  We have
  \begin{align*}
    G\cdot s\bigl(\bigl(a(t\cdot b)\bigr)\cdot h\bigr) &=G\cdot
    s\bigl((a\cdot h)(t\cdot b\cdot h)\bigr) =G\cdot s(t\cdot b\cdot
    h) =G\cdot s(b\cdot h) \\&=G\cdot r(c),
  \end{align*}
  so
  \begin{equation*}
    \bigl(a(t\cdot b)\bigr)\cdot (h,G\cdot c)
  \end{equation*}
  is defined, and then we have
  \begin{equation*}
    \bigl((a\cdot H,t)\cdot b\bigr)\cdot (h,G\cdot c)
    =\bigl(\bigl(a(t\cdot b)\bigr)\cdot h\bigr)(t'\cdot c),
  \end{equation*}
  where $t'\in G$ is unique with
  \begin{align*}
    t'\cdot r(c) =s\bigl(\bigl(a(t\cdot b)\bigr)\cdot h\bigr)
    =s\bigl((a\cdot h)(t\cdot b\cdot h)\bigr)=s(t\cdot b\cdot h)
    =t\cdot s(b\cdot h).
  \end{align*}
  But we have arranged for $r(c)=s(b\cdot h)$, so in fact we have
  $t'=t$, and hence
  \begin{equation*}
    \bigl((a\cdot H,t)\cdot b\bigr)\cdot (h,G\cdot c) =(a\cdot h)(t\cdot
    b\cdot h)(t\cdot c)
  \end{equation*}
  Similarly, we have
  \begin{equation*}
    r\bigl((t\cdot\bigl((b\cdot h)c\bigr)\bigr)\cdot H =s(a)\cdot H,
  \end{equation*}
  so
  \begin{align*}
    (a\cdot H,t)\cdot\bigl(b\cdot(h,G\cdot c)\bigr) &=(a\cdot
    H,t)\cdot \bigl((b\cdot h)c\bigr) =(a\cdot
    h)\bigl(t\cdot\bigl((b\cdot h)c\bigr)\bigr) \\&=(a\cdot h)(t\cdot
    b\cdot h)(t\cdot c),
  \end{align*}
  and we have shown that the actions of $(\AA/H)\rtimes G$ and
  $H\ltimes (G\under\AA)$ commute.
\end{step}

\begin{step}
  Item (b) of \cite[Definition~6.1]{mw:fell} has four parts,
  concerning the left and right inner products.  Before we begin, we
  first check that the left and right inner products in
  \thmref{symmetric action} are well-defined, and by symmetry it
  suffices to check the left-hand inner product.  Let $(a,b)\in
  \AA*_\sigma\AA$.  Then $\sigma(a)=\sigma(b)$, and $\bigl(e,G\cdot
  s(a)\bigr) =\bigl(e,G\cdot s(b)\bigr)$, so that there is a unique
  $t\in G$ such that $ s(a)=t\cdot s(b)$.  Thus
  \begin{equation*}
    s(a)=r\bigl((t\cdot b)^*\bigr),
  \end{equation*}
  and therefore the inner product is well-defined.

  We proceed to the first part of \cite[Definition~6.1 (b)]{mw:fell},
  namely
  \begin{equation*}
    p_{\XX/H\rtimes G}\left({}_L\<a,b\>\right) ={}_{\XX/H\rtimes
      G}\bigl[p(a),p(b)\bigr] \righttext{for}(a,b)\in \AA*_\sigma \AA,
  \end{equation*}
  where
  \begin{equation*}
    p_{\XX/H\rtimes G}:\AA/H\rtimes G\to \XX/H\rtimes G
  \end{equation*}
  is the bundle projection.  According to the definition of the left
  inner product, we have
  \begin{align*}
    p\bigl({}_L\<a,b\>\bigr) &=p\bigl(\bigl(a(t\cdot b^*)\cdot
    H,t\bigr)\bigr) =\bigl(p\bigl(a(t\cdot b^*)\bigr)\cdot H,t\bigr)\\
    &=\bigl(p(a)\bigl(t\cdot p(b)\inv\bigr)\cdot H,t\bigr)
    ={}_L[p(a),p(b)],
  \end{align*}
  by \eqref{left bracket}.
\end{step}

\begin{step}
  The second part of \cite[Definition~6.1 (b)]{mw:fell} is
  \begin{equation*} {}_L\<a,b\>^*={}_L\<b,a\> \and
    \<a,b\>_R^*=\<b,a\>_R,
  \end{equation*}
  and again by symmetry it suffices to prove the first.  Let
  $a,b\in\AA$ with $\sigma(a)=\sigma(b)$, and let $t\in G$ be unique
  such that $ s(a)=t\cdot s(b)$.  Then
  \begin{align*} {}_L\<a,b\>^* &=\bigl(a(t\cdot b^*)\cdot H,t\bigr)^*
    =\bigl(t\inv\cdot\bigl(a(t\cdot b^*)\bigr)^*\cdot H,t\bigr)
    \\&=\bigl(t\inv\cdot\bigl((t\cdot b)a^*\bigr)^*\cdot H,t\bigr)
    =\bigl(b(t\inv\cdot a^*)\cdot H,t\inv) ={}_L\<b,a\>.
  \end{align*}
\end{step}

\begin{step}
  The third part of \cite[Definition~6.1 (b)]{mw:fell} is
  \begin{equation*} {}_L\bigl\<(a\cdot H,t)\cdot b,c\bigr\> =(a\cdot
    H,t){}_L\<b,c\>,
  \end{equation*}
  and a similar equality involving $\<\cdot,\cdot\>_R$, but by
  symmetry it suffices to show it for the left inner product: let
  \begin{equation*}
    s\bigl(p_{\AA/H\rtimes G}(a\cdot H,t)\bigr)=\rho(b) \and
    \sigma(b)=\sigma(c),
  \end{equation*}
  so that
  \begin{equation*}
    s(a)\cdot H=t\cdot r(b)\cdot H \midtext{and} G\cdot s(b)=G\cdot
    s(c).
  \end{equation*}
  Adjust $a$ within its orbit $a\cdot H$ so that $s(a)=t\cdot r(b)$,
  and choose the unique $t'\in G$ such that $s(b)=t'\cdot s(c)$.  Then
  \begin{align*}
    (a\cdot H,t){}_L\<b,c\> &=(a\cdot H,t)\bigl(b(t'\cdot c^*)\cdot
    H,t'\bigr) =\bigl(a\bigl(t\cdot \bigl(b(t'\cdot
    c^*)\bigr)\bigr)\cdot H,tt'\bigr) \\&=\bigl(a(t\cdot b)(tt'\cdot
    c)\cdot H,tt'\bigr) ={}_L\bigl\<a(t\cdot b),c\bigr\>
    ={}_L\bigl\<(a\cdot H,t)\cdot b,c\bigr\>
  \end{align*}
\end{step}

\begin{step}
  The fourth part of \cite[Definition~6.1 (b)]{mw:fell} is
  \begin{equation*} {}_L\<a,b\>\cdot c=a\cdot \<b,c\>_R.
  \end{equation*}
  More precisely, we need to show that for $a,b,c\in\AA$, if both
  ${}_L\<a,b\>$ and $\<b,c\>_R$ are defined, then so are
  \begin{equation*} {}_L\<a,b\>\cdot c \and a\cdot \<b,c\>_R,
  \end{equation*}
  and they are equal.  Thus, we are assuming that
  \begin{equation*}
    \sigma(a)=\sigma(b) \midtext{and} \rho(b)=\rho(c),
  \end{equation*}
  which entails that
  \begin{equation*}
    G\cdot s(a)=G\cdot s(b) \midtext{and} r(b)\cdot H=r(c)\cdot H.
  \end{equation*}
  Choose the unique $t\in G$ and $h\in H$ such that $ s(a)=t\cdot
  s(b)$ and $ r(b)\cdot h=r(c)$.  Then
  \begin{align*} {}_L\<a,b\>\cdot c &=\bigl(a(t\cdot b^*)\cdot
    H,t\bigr)\cdot c =\bigl(\bigl(a(t\cdot b^*)\bigr)\cdot
    h\bigr)(t\cdot c) =(a\cdot h)(t\cdot b^*\cdot h)(t\cdot c)
    \\&=(a\cdot h)\bigl(t\cdot \bigl((b^*\cdot h)c\bigr)\bigr) =a\cdot
    \bigl(h,G\cdot (b^*\cdot h)c\bigr) =a\cdot \<b,c\>_R,
  \end{align*}
  where the second equality is justified by
  \begin{align*}
    s\bigl((a(t\cdot b^*)\bigr)\cdot h\bigr) &=s\bigl((a\cdot
    h)(t\cdot b^*\cdot h)\bigr) =s(t\cdot b^*\cdot h) =t\cdot
    s(b^*\cdot h) \\&=t\cdot r(b\cdot h) =t\cdot r(c),
  \end{align*}
  and a similar computation justifies the fifth equality.
\end{step}

\begin{step}
  Finally, item (c) of \cite[Definition~6.1 (c)]{mw:fell} is that that
  the operations (i)--(iv) make each fibre $A(x)$ of the Banach bundle
  $\AA$ into an imprimitivity bimodule between the corresponding
  fibres
  \begin{equation*}
    A(r(x))\cdot H\times\{e\} \and \{e\}\times G\cdot A(s(x))
  \end{equation*}
  of the Fell bundles
  \begin{equation*}
    \AA/H\rtimes G \and H\ltimes G\under\AA,
  \end{equation*}
  respectively.  For this we just have to observe that, in view of the
  obvious isomorphisms
  \begin{equation*}
    A(r(x))\cdot H\times \{e\}\cong A(r(x)) \midtext{and} \{e\}\times
    G\cdot A(s(x))\cong A(s(x)),
  \end{equation*}
  our inner products and actions \emph{coincide} with those on the
  $A(r(x))-A(s(x))$ imprimitivity bimodule $A(x)$.  \qedhere
\end{step}
\end{proof}

We will need the special case of \thmref{symmetric action} where one
group is trivial, and by symmetry it suffices to consider the case
where the group $H$ is trivial:

\begin{cor}
  \label{one-sided action}
  If a locally compact group $G$ acts freely and properly on the left
  of a Fell bundle $p:\AA\to\XX$ over a locally compact groupoid, then
  $\AA$ becomes a $\AA\rtimes G-G\under \AA$ equivalence in the
  following way:
  \begin{compactenum}
  \item $\AA\rtimes G$ acts on the left of $\AA$ by
    \begin{equation*}
      (a,t)\cdot b=a(t\cdot b) \quad\text{if}\quad s(a)=t\cdot r(b);
    \end{equation*}

  \item the left inner product is given by
    \begin{equation*} {}_L\<a,b\> =\bigl(a(t\cdot b^*),t\bigr)
      \quad\text{if}\quad G\cdot s(a)=G\cdot s(b),
    \end{equation*}
    where $t$ is the unique element of $G$ such that $ s(a)=t\cdot
    s(b)$;

  \item $G\under\AA$ acts on the right of $\AA$ by
    \begin{equation*}
      a\cdot (G\cdot b)=ab \quad\text{if}\quad s(a)=r(b);
    \end{equation*}

  \item the right inner product is given by
    \begin{equation*}
      \<a,b\>_R =G\cdot a^*b \quad\text{if}\quad r(a)=r(b).
    \end{equation*}
  \end{compactenum}
\end{cor}

\begin{rem}
  Of course, if we have an action of a group $H$ on the right of $\AA$
  instead of $G$ acting on the left, the corresponding equivalence
  will be
  \begin{equation*}
    \AA/H\sim H\ltimes \AA.
  \end{equation*}
\end{rem}

\begin{rem}
  Note that (iii) has been expressed in ``cleaned-up'' form; in
  general, if $a,b\in\AA$ then the right-module product $a\cdot
  (G\cdot b)$ is defined if and only if
  \begin{equation*}
    G\cdot s(a)=G\cdot r(b),
  \end{equation*}
  and then we have
  \begin{equation*}
    a\cdot (G\cdot b)=a(t\cdot b)
  \end{equation*}
  where $t$ is the unique element of $G$ such that $ s(a)=t\cdot
  r(b)$.
\end{rem}

It will be convenient to have another version of \corref{one-sided
  action}, recorded as \corref{one-sided transformation} below, with
the Fell bundle $\AA\to\XX$ replaced with the isomorphic
transformation bundle.


\begin{cor}
  \label{one-sided transformation}
  Let $p:\BB\to\YY$ be a Fell bundle over a locally compact groupoid,
  let $\YY$ act on the left of a locally compact Hausdorff space
  $\Omega$, and let the associated fibre map $\rho:\Omega\to
  \YY^{(0)}$ be a principal $G$-bundle, with the locally compact group
  $G$ acting on the left of $\Omega$.  Further assume that the actions
  of $\YY$ and $G$ on $\Omega$ commute.  Then the transformation
  bundle $\BB*\Omega\to \YY*\Omega$ becomes a $(\BB*\Omega)\rtimes
  G-\BB$ equivalence in the following way:
  \begin{enumerate}
  \item $(\BB*\Omega)\rtimes G$ acts on the left of $\BB*\Omega$ by
    \begin{equation*}
      (b,p(c)\cdot t\cdot u,t)\cdot (c,u)=(bc,t\cdot u)
      \quad\text{if}\quad s(b)=r(c).
    \end{equation*}

  \item the left inner product is given by
    \begin{equation*} {}_L\bigl\<(b,t\cdot u),(c,u)\bigr\>
      =\bigl(bc^*,t\cdot p(c)\cdot u,t\bigr) \quad\text{if}\quad
      s(b)=s(c);
    \end{equation*}

  \item $\BB$ acts on the right of $\BB*\Omega$ by
    \begin{equation*}
      (b,u)\cdot c=(bc,p(c)\inv\cdot u) \quad\text{if}\quad s(b)=r(c);
    \end{equation*}

  \item the right inner product is given by
    \begin{equation*}
      \bigl\<(b,p(b^*c)\cdot u),(c,u)\bigr\>_\BB =b^*c
      \quad\text{if}\quad r(b)=r(c).
    \end{equation*}
  \end{enumerate}
\end{cor}

Note that the fibring maps
\begin{equation*}
  \rho:\YY*\Omega\to (\YY^{(0)}*\Omega)\times \{e\} \midtext{and}
  \sigma:\YY*\Omega\to \YY^{(0)}
\end{equation*}
associated to the actions of $(\YY*\Omega)\rtimes G$ and $\YY$ are
given by
\begin{equation*}
  \rho(y,u)=(r(y),y\cdot u,e) \midtext{and} \sigma(y,u)=s(y).
\end{equation*}
Also note that every section $f\in \Gamma_c((\BB*\Omega)\rtimes G)$ is
uniquely of the form
\begin{equation*}
  f(y,u,s)=\bigl(f_1(y,u,s),u,s\bigr),
\end{equation*}
where $f_1:(\YY*\Omega)\rtimes G\to \BB$ is continuous with compact
support and satisfies
\begin{equation*}
  f_1(y,u,s)\in B(y).
\end{equation*}

Once we have the Fell-bundle equivalence \thmref{symmetric action}, by
\cite[Theorem~6.4]{mw:fell} we have a morita equivalence:

\begin{cor}\label{symmetric Morita}
  With the hypotheses of \thmref{symmetric action} we have a Morita
  equivalence
  \begin{equation*}
    C^*(\AA/H)\rtimes_\alpha G\sim C^*(G\under\AA)\rtimes_\beta H,
  \end{equation*}
  where $\alpha$ and $\beta$ are the associated actions of $G$ and $H$
  on the respective Fell-bundle $C^*$-algebras.
\end{cor}

\begin{proof}
  This follows from \thmref{symmetric action},
  \cite[Theorem~6.4]{mw:fell}, and \cite[Theorem~7.1]{kmqw1}.
\end{proof}

Similarly, we have a one-sided Morita equivalence:

\begin{cor}\label{one-sided Morita}
  With the hypotheses of \corref{one-sided action} we have a Morita
  equivalence
  \begin{equation*}
    C^*(\AA)\rtimes_\alpha G\sim C^*(G\under\AA),
  \end{equation*}
  where $\alpha$ is as in \corref{symmetric Morita}.
\end{cor}

And here is the version for transformation bundles:

\begin{cor}\label{one-sided transformation Morita}
  With the hypotheses of \corref{one-sided transformation} we have a
  Morita equivalence
  \begin{equation*}
    C^*(\BB*\Omega)\rtimes_\alpha G\sim C^*(\BB),
  \end{equation*}
  where $\alpha$ is as in \corref{one-sided Morita}.
\end{cor}

\subsection*{Special case: $C^*$-bundles}

Here we specialize to the case where $\XX=X$ is a \emph{space} and
$\AA$ is just a $C^*$-bundle over $X$, so that
$C^*(\AA)=\Gamma_0(\AA)$.  Then the orbit bundle is the $C^*$-bundle
$G\under\AA\to G\under X$.

\begin{prop}
  \label{C*-bundle}
  If a locally compact group $G$ acts freely and properly on a
  $C^*$-bundle $\AA\to X$ over a locally compact Hausdorff space $X$,
  then we have a Morita equivalence
  \begin{equation*}
    \Gamma_0(\AA)\rtimes_\alpha G\sim \Gamma_0(G\under\AA),
  \end{equation*}
  where $\alpha:G\to\aut\Gamma_0(\AA)$ is the associated action.
\end{prop}

\subsection*{Special case: coaction-crossed products}
We start with a Fell bundle $\BB$ over a locally compact group $G$,
then form the transformation Fell bundle $\AA=\BB\times G$ over the
transformation groupoid $G\times_{\lt} G$, where $G$ acts on itself by
left translation.  Then $G$ acts freely and properly on $\BB\times G$
by right translation in the second coordinate:
\begin{equation*}
  t\cdot (b,s)=(b,st\inv).
\end{equation*}
We identify the orbit bundle $G\under(\BB\times G)$ with $\BB$, and
the orbit groupoid $(G\times_{\lt} G)/G$ with the group $G$.

By \cite[Theorem~5.1]{kmqw1} there are a maximal coaction $\delta$ of
$G$ on $C^*(\BB)$ and an $\alpha-\hat\delta$ equivariant isomorphism
\begin{equation*}
  C^*(\BB\times G)\cong C^*(\BB)\rtimes_\delta G.
\end{equation*}
Thus \corref{one-sided transformation Morita} reduces in this case to
cross-product duality for maximal coactions:
\begin{equation*}
  C^*(\BB)\rtimes_\delta G\rtimes_{\hat\delta} G\sim C^*(\BB).
\end{equation*}
In the even more special case where $\BB=\C\times G$ is the trivial
line bundle, we recover the well-known Morita equivalence
\begin{equation*}
  \KK(L^2(G))\rtimes_{\ad\rho} G\sim C^*(G),
\end{equation*}
where $\rho$ denotes the right regular representation of $G$.

\section{Raeburn's Symmetric Imprimitivity Theorem}

We recover Raeburn's symmetric imprimitivity theorem
\cite[Theorem~1.1]{rae:symmetric} as a corollary to \corref{symmetric
  Morita}:

\begin{cor}
  Suppose that locally compact groups $G$ and $H$ act freely and properly on the left and right, respectively, of a locally
  compact Hausdorff space $X$, and suppose the actions commute. Suppose, also, that $\sigma$ and $\tau$ are
  commuting actions of $G$ and $H$, respectively, on a $C^*$-algebra
  $B$, then we have a Morita equivalence
  \begin{equation*}
  \ind_H^XB\rtimes_{\ind\tau} G\sim \ind_G^XB\rtimes_{\ind\sigma} H.
  \end{equation*}
\end{cor}

\begin{proof}
  Recall from \cite{rae:symmetric} that we have induced actions
  $\ind\tau$ and $\ind\sigma$ of $G$ and $H$ on the induced
  $C^*$-algebras $\ind_H^X B$ and $\ind_G^X B$, respectively.  We aim
  to apply \thmref{symmetric Morita} with $\AA=B\times X$, $\XX=X$,
  and $G$ and $H$ acting diagonally on $B\times X$ (with the right
  $H$-action given by $(b,x)\cdot h=(\tau_h\inv(b),x\cdot h)$).  The
  orbit Fell bundle $\AA/H\to \XX/H$ is the $C^*$-bundle $(B\times
  X)/H\to X/H$, and similarly for $G\under (B\times X)\to G\under X$.
  Let $\alpha$ and $\beta$ denote the associated actions of $G$ and
  $H$ on $\Gamma_0((B\times X)/H)$ and $\Gamma_0(G\under (B\times
  X))$, respectively.  By \thmref{symmetric Morita} we can finish by
  recalling from the standard theory of induced algebras that there
  are equivariant isomorphsims:
  \begin{align*}
    \bigl(\Gamma_0\bigl((B\times X)/H\bigr),\alpha\bigr) &\cong
    (\ind_H^XB,\ind\tau)
    \\
    \bigl(\Gamma_0(\bigl(G\under (B\times X)\bigr),\beta\bigr) &\cong
    (\ind_G^XB,\ind\sigma).
  \end{align*}
  For example, a suitable isomorphism $\theta:\ind_H^X B\to
  \Gamma_0((B\times X)/H)$ is given by $\theta(f)(x\cdot
  H)=(f(x),x)\cdot H$.
\end{proof}


\begin{appendix}
  \section{Bundle constructions}
  \label{appendix}

  In this appendix, we will review several constructions involving
  Fell bundles over groupoids, some from \cite{kmqw1} and some new,
  that we need in the body of the paper.

  \subsection{Transformation groupoids and bundles}
  \label{sec:transf-group-bundl}

  Suppose a locally compact groupoid $\XX$ acts on the left of a
  locally compact Hausdorff space $\Omega$. Let $\rho:\Omega\to
  \XX^{(0)}$ be the associated fibre map, and let
  \begin{equation*}
    \XX*\Omega=\{(x,u)\in \XX\times \Omega:s(x)=\rho(u)\}
  \end{equation*}
  be the fibre product. Then $\XX*\Omega$ is locally compact
  Hausdorff, and becomes the
  \emph{transformation groupoid} with multiplication
  \begin{equation*}
    (x,y\cdot u)(y,u)=(xy,u).
  \end{equation*}
Thus
\begin{equation*}
  (x,u)\inv=(x\inv,x\cdot u),\quad s(x,u)=(s(x),u)\quad\text{and}
  \quad r(x,u)=(r(x),x\cdot u).
\end{equation*}

\begin{rem*}
  We've made a choice of convention here for the transformation
  groupoid --- another fairly common choice involves writing $\XX$ and
  $\Omega$ in the opposite order.  But we will always use the above
  convention.
\end{rem*}

Note that the coordinate projection $\pi_1:\XX*\Omega\to\XX$ defined
by
\begin{equation*}
  \pi_1(x,u)=x
\end{equation*}
is a groupoid homomorphism.  If now $p:\AA\to\XX$ is a Fell bundle,
then the associated \emph{transformation Fell bundle} has total space
\begin{equation*}
  \AA*\Omega:=\{(a,u)\in \AA\times \Omega:s(a)=\rho(u)\},
\end{equation*}
base groupoid $\XX*\Omega$, bundle projection
\begin{equation*}
  p(a,u)=(p(a),u),
\end{equation*}
multiplication
\begin{equation*}
  (a,p(b)\cdot u)(b,u)=(ab,u),
\end{equation*}
and involution
\begin{equation*}
  (a,u)^*=(a^*,p(a)\cdot u).
\end{equation*}
The easiest way to see that this is a Fell bundle is to note that it
is isomorphic to the pullback \cite[Lemma~1.1]{kmqw1} $\pi_1^*\AA$ via
the map
\begin{equation*}
  (a,u)\mapsto (a,p(a)\cdot u).
\end{equation*}

It is not hard to check that we get a left Haar system on $\XX*\Omega$
via\footnote{Note that the formula in \cite{MuhlyCBMS} for the Haar
  system on the 
transformation groupoid is incorrect.}
\begin{equation*}
  \int_{\XX*\Omega} f(y,v)\,d\lambda^{r(x,u)}(y,v) =\int_\XX f(y,y\inv
  x\cdot u)\,d\lambda^{r(x)}(y).
\end{equation*}

Thus $C_c(\XX*\Omega)$ has convolution and involution given by
\begin{equation*}
  (f*g)(x,u)=\int_\XX f(y,y\inv x \cdot u)g(y\inv
  x,u)\,d\lambda^{r(x)}(y)
\end{equation*}
and
\begin{equation*}
  f^*(x,u)=\bar{f(x\inv,x\cdot u)}.
\end{equation*}

Every section $f\in \Gamma_c(\AA*\Omega)$ is of the form
\begin{equation*}
  f(x,u)=(f_1(x,u),u)
\end{equation*}
for some function $f_1\in C_c(\XX*\Omega,\AA)$ satisfying
\begin{equation*}
  f_1(x,u)\in A(x)\righttext{for}(x,u)\in \XX*\Omega.
\end{equation*}
Convolution and involution in the section algebra
$\Gamma_c(\AA*\Omega)$ are determined by
\begin{equation*}
  (f*g)_1(x,u)=\int_\XX f_1(y,y\inv x\cdot u)g_1(y\inv
  x,u)\,d\lambda^{r(x)}(y)
\end{equation*}
and
\begin{equation*}
  (f^*)_1(x,u)=f_1(x\inv,x\cdot u)^*,
\end{equation*}
and we will simplify the notation by writing $f^*_1$ for $(f^*)_1$.

\subsection{Semidirect products}

Let $G$ be a locally compact group and $\XX$ be a locally compact
groupoid.  Recall from \cite[Section~6]{kmqw1} that an \emph{action by
automorphisms} 
of $G$ on (the left of) $\XX$ is a continuous map
\begin{equation*}
  (s,x)\mapsto s\cdot x:G\times\XX\to\XX
\end{equation*}
such that
\begin{itemize}
\item for each $s\in G$, the map $x\mapsto s\cdot x$ is an
  automorphism of the groupoid $\XX$, and

\item $s\cdot (t\cdot x)=(st)\cdot x$ for all $s,t\in G$ and
  $x\in\XX$,
\end{itemize}
and that the associated \emph{semidirect-product} groupoid $\XX\rtimes
G$ comprises the Cartesian product $\XX\times G$ with multiplication
\begin{equation*}
  (x,s)(y,t)=\bigl(x(s\cdot y),st\bigr) \righttext{if $x$ and $s\cdot y$
    are composable.}
\end{equation*}
Thus the range and source maps are given by
\begin{equation*}
  r(x,t)=(r(x),e)\and s(x,t)=\bigl(t\inv\cdot s(x),e\bigr),
\end{equation*}
and the inverse is given by
\begin{equation*}
  (x,s)\inv=(s\inv\cdot x\inv,s\inv).
\end{equation*}

\begin{warn}
  Frequently we will have a group (or a groupoid) acting on a
groupoid $\XX$ \emph{as a space} --- so then the action on $\XX$ is
just by homeomorphisms, not automorphisms --- and to avoid confusion
we will usually emphasize the particular type of action on $\XX$: $G$
either \emph{acts on $\XX$ by automorphisms} or \emph{acts on the
  space~$\XX$}.
\end{warn}

Recall from \cite[Section~6]{kmqw1} that, when $G$ acts on $\XX$ by
automorphisms, in order to get a Haar system on the semidirect-product
groupoid $\XX\rtimes G$ we need to assume that the action is
\emph{invariant} in the sense that
\begin{equation*}
  \int_\XX f(s\cdot x)\,d\lambda^{u}(x)=\int_\XX f(x)\,d\lambda^{s\cdot
    u}(x),
\end{equation*}
and then
\begin{equation*}
  d\lambda^{(u,e)}(x,s)=d\lambda^u(x)\,ds
\end{equation*}
is a Haar system on $\XX\rtimes G$.

Also recall from \cite{kmqw1} that if $p:\AA\to\XX$ is a Fell bundle,
then an \emph{action} of $G$ on (the left of) $\AA$ (by automorphisms)
consists of an
action of $G$ on $\XX$ and a continuous map $G\times\AA\to\AA$ such
that
\begin{itemize}
\item for each $s\in G$, the map $a\mapsto s\cdot a$ is an
  automorphism of the Fell bundle $\AA$ (which entails $p(s\cdot
  a)=s\cdot p(a)$), and

\item $s\cdot (t\cdot a)=(st)\cdot a$ for all $s,t\in G$ and
  $a\in\AA$,
\end{itemize}
and that the associated \emph{semidirect-product} Fell bundle
$\AA\rtimes G$ comprises the Cartesian product $\AA\times G$ with
bundle projection
\begin{equation*}
  p(a,s)=(p(a),s),
\end{equation*}
multiplication
\begin{equation*}
  (a,s)(b,t)=\bigl(a(s\cdot b),st\bigr) \righttext{if $a$ and $s\cdot b$
    are composable,}
\end{equation*}
and involution
\begin{equation*}
  (a,s)^*=(s\inv\cdot a^*,s\inv).
\end{equation*}

Thus, convolution and involution on $C_c(\XX\rtimes G)$ are given by
\begin{equation*}
  (\phi*\psi)(x,s)=\int_\XX\int_G \phi(y,t)\psi\bigl(t\inv\cdot (y\inv
  x),t\inv s\bigr)\,dt\,d\lambda^{r(x)}(y)
\end{equation*}
and
\begin{equation*}
  \phi^*(x,s)=\bar{\phi(s\inv\cdot x\inv,s\inv)}.
\end{equation*}
If $G$ acts by automorphisms on a Fell bundle $p:\AA\to\XX$ then every
section $f\in \Gamma_c(\AA\rtimes G)$ is of the form
\begin{equation*}
  f(x,s)=(f_1(x,s),s)
\end{equation*}
for some function $f_1\in C_c(\XX\times G,\AA)$ satisfying
\begin{equation*}
  f_1(x,s)\in A(x)\righttext{for}(x,s)\in \XX\times G.
\end{equation*}
Convolution and involution in the section algebra $\Gamma_c(\AA\rtimes
G)$ are determined by
\begin{equation*}
  (f*f')_1(x,s)=\int_\XX\int_G f_1(y,t)f'_1\bigl(t\inv\cdot (y\inv
  x),t\inv s\bigr)\,dt\,d\lambda^{r(x)}(y)
\end{equation*}
and
\begin{equation*}
  f^*_1(x,s)=f_1(s\inv\cdot x\inv,s\inv)^*.
\end{equation*}

Similarly, if $G$ acts on the right of $\XX$ rather than the left, the
semidirect product $G\ltimes\XX$ has multiplication
\begin{equation*}
  (s,x)(t,y)=\bigl(st,(x\cdot t)y\bigr),
\end{equation*}
range and source maps
\begin{equation*}
  r(t,x)=(e,r(x)\cdot t\inv) \and s(t,x)=(e,s(x)),
\end{equation*}
and inverse
\begin{equation*}
  (s,x)\inv=(s\inv,x\inv\cdot s\inv).
\end{equation*}
%

\subsection{Semidirect-product actions}

\begin{defn}
  Let $G$ be a locally compact group, $\XX$ be a locally compact
  groupoid, and $\Omega$ be a locally compact Hausdorff space.
  Suppose $G$ acts on $\XX$ by groupoid automorphisms, and that both
  $G$ and $\XX$ act on $\Omega$, with all actions being on the left.
  We say that the actions of $G$ and $\XX$ on $\Omega$ are
  \emph{covariant} if
  \begin{equation*}
    s\cdot (x\cdot u)=(s\cdot x)\cdot (s\cdot u)\righttext{for
      all}(x,u)\in \XX*\Omega.
  \end{equation*}
  Similarly for covariant right actions.
\end{defn}

\begin{lem}
  \label{semidirect action}
  Let $G$ be a group, $\XX$ be a groupoid, and $\Omega$ be a space.
  Suppose $G$ acts on $\XX$ by automorphisms, and that both $G$ and
  $\XX$ act on $\Omega$, with all actions being on the left.  If the
  actions of $G$ and $\XX$ on $\Omega$ are covariant, then the
  semidirect-product groupoid $\XX\rtimes G$ acts on $\Omega$ by
  \begin{equation}
    \label{semidirect act}
    (x,t)\cdot u=x\cdot (t\cdot u)
    \quad\text{if}\quad s(x)=\rho(t\cdot u),
  \end{equation}
  where $\rho:\Omega\to \XX^{(0)}$ is the associated fibre map.
\end{lem}

\begin{proof}
  Define
  \begin{equation*}
    \rho':\Omega\to (\XX\rtimes G)^{(0)}=\XX^{(0)}\times \{e\}
  \end{equation*}
  by
  \begin{equation*}
    \rho'(u)=(\rho(u),e).
  \end{equation*}
  Then $\rho'$ is clearly continuous and open.  Moreover
  \begin{equation*}
    s(x,t)=\rho'(u) \midtext{if and only if} s(x)=\rho(t\cdot u),
  \end{equation*}
  so the definition \eqref{semidirect act} of $(x,t)\cdot u$ is
  well-defined.  It is routine to check the axioms for an action, even
  the continuity: if
  \begin{equation*}
    \bigl((x_i,t_i),u_i\bigr)\to \bigl((x,t),u\bigr)
    \righttext{in}(\XX\rtimes )*\Omega,
  \end{equation*}
  then
  \begin{align*}
    (x_i,t_i)\cdot u_i =x_i\cdot (t_i\cdot u_i) \to x\cdot (t\cdot u)
    =(x,t)\cdot u
  \end{align*}
  because $x_i\to x$ and $t_i\cdot u_i\to t\cdot u$.
\end{proof}

For convenient reference, we record the corresponding result for right
actions:

\begin{cor}
  \label{semidirect right action}
  If a group $H$ acts on a groupoid $\XX$, and $H$ and $\XX$ act
  covariantly on a space $\Omega$, all actions being on the right,
  then $H\ltimes \XX$ acts on the right of $\Omega$ by
  \begin{equation*}
    u\cdot (h,x)=(u\cdot h)\cdot x \quad\text{if}\quad\sigma(u)=r(x\cdot
    h\inv),
  \end{equation*}
  where $\sigma:\Omega\to\XX^{(0)}$ is the associated fibre map.
\end{cor}

\begin{defn}
  Let $G$ be a group, let $p:\AA\to\XX$ be a Fell bundle over a
  groupoid, and let $q:\EE\to \Omega$ be a Banach bundle.  Suppose $G$
  acts on $\AA$ by Fell-bundle automorphisms, and that both $G$ and
  $\AA$ act on $\EE$ (with $G$ acting by isometric isomorphisms, and
  $\AA$ acting as in \cite{mw:fell}), with all actions being on the
  left.  We say that the actions of $G$ and $\AA$ on $\EE$ are
  \emph{covariant} if
  \begin{equation*}
    s\cdot (a\cdot e)=(s\cdot a)\cdot (s\cdot e)
    \quad\text{if}\quad s(a)=\rho(q(e)).
  \end{equation*}
\end{defn}

\begin{cor}
  \label{semidirect bundle action}
  Let $G$ be a group, let $p:\AA\to\XX$ be a Fell bundle over a
  groupoid, and let $q:\EE\to \Omega$ be a Banach bundle.  Suppose $G$
  acts on $\AA$ by Fell-bundle automorphisms, and that both $G$ and
  $\AA$ act on $\EE$ \($G$ acting by isometric isomorphisms and $\AA$
  acting as in \cite{mw:fell}\), with all actions being on the left.
  If the actions of $G$ and $\AA$ on $\EE$ are covariant, then the
  semidirect-product Fell bundle $\AA\rtimes G$ acts on $\EE$ by
  \begin{equation}
    \label{semidirect bundle act}
    (a,t)\cdot e=a\cdot (t\cdot e)
    \quad\text{if}\quad s(a)=\rho(q(t\cdot e)).
  \end{equation}
\end{cor}

\begin{proof}
  Recall from \lemref{semidirect action} that $\XX\rtimes G$ acts on
  the space $\Omega$, and that the associated fiber map
  $\rho':\Omega\to \{e\}\times\XX^{(0)}$ is
  \begin{equation*}
    \rho'(u)=(\rho(u),e),
  \end{equation*}
  where $\rho:\Omega\to\XX^{(0)}$ is the fibre map associated to the
  action of $\XX$ on $\Omega$.  We have
  \begin{equation*}
    s(p(a,t))=\rho'(q(e))
  \end{equation*}
  if and only if
  \begin{equation*}
    s(a)=\rho(q(t\cdot e)).
  \end{equation*}
  It is routine to check the axioms for an action of a Fell bundle on
  a Banach bundle.
\end{proof}

\subsection{Quotients}

We want to know that if a group $H$ acts freely by automorphisms on a
groupoid $\XX$, then the orbit space $\XX/H$ is a groupoid, and
similarly if $H$ acts on a Fell bundle $\AA$ over $\XX$.

Of course we want all this to be topological.  Thus we take $H$ and
$\XX$ to be locally compact Hausdorff.  To ensure that $\XX/H$ also
has these properties, we require that the action be free and proper.

\begin{prop}
  \label{orbit groupoid}
  Let $H$ be a locally compact group and $\XX$ a locally compact
  Hausdorff groupoid.  Suppose that $H$ acts freely and properly on
  the right of $\XX$ by automorphisms.  Then the orbit space $\XX/H$
  becomes a locally compact Hausdorff groupoid, called an \emph{orbit
    groupoid} or a \emph{quotient groupoid}, with multiplication
  \begin{equation}
    \label{orbit multiply}
    (x\cdot H)(y\cdot H)=(xy)\cdot H
    \righttext{whenever}s(x)=r(y).
  \end{equation}

  Moreover, if the action of $H$ on $\XX$ is invariant in the sense of
  \cite[Definition~6.3]{kmqw1}, then there is a Haar system
  $\dot\lambda$ on $\XX/H$ given by
  \begin{equation}\label{Haar}
    \int_{\XX/H} f(x\cdot H)\,d\dot\lambda^{u\cdot H}(x\cdot H)
    =\int_\XX f(x\cdot H)\,d\lambda^u(x).
  \end{equation}
\end{prop}

\begin{proof}
  Because $H$ acts freely, if $s(x)\cdot H=r(y)\cdot H$ then there is
  a unique $h\in H$ such that $s(x\cdot h)=r(y)$. But a moment's
  though reveals that somewhat more is true: the set
  \begin{equation*}
    \{zw:z\in x\cdot H,w\in y\cdot H,s(z)=r(w)\}
  \end{equation*}
  comprises a single orbit, represented by any such product $zw$. In
  particular, if we choose $x$ within its orbit $x\cdot H$ so that
  $s(x)=r(y)$, then the above set of products coincides with the orbit
  $(xy)\cdot H$.  Thus, replacing $x$ by $x\cdot h$ we see that the
  multiplication \eqref{orbit multiply} is well-defined.  It is
  routine to check that the formula \eqref{orbit multiply} does make
  $\XX/H$ into a groupoid, and it is locally compact and Hausdorff
  because $H$ acts properly. It remains to verify continuity of
  multiplication and inversion in $\XX/H$.

  Let
  \begin{equation*}
    x_i\cdot H\to x\cdot H\and y_i\cdot H\to y\cdot H
    \righttext{in}\XX/H,
  \end{equation*}
  and assume that
  \begin{equation*}
    s(x_i\cdot H)=r(y_i\cdot H)\all i.
  \end{equation*}
  We will show that $s(x\cdot H)=r(y\cdot H)$ and
  \begin{equation}
    \label{product converge}
    (x_i\cdot H)(y_i\cdot H)\to (x\cdot H)(y\cdot H).
  \end{equation}
  It suffices to show that every subnet has a subnet satisfying
  \eqref{product converge}, and since any subnet will continue satisfy
  our hypotheses, it suffices to show that some subnet satisfies
  \eqref{product converge}.

  Note that the range and source maps on $\XX/H$ are continuous and
  open, since the range and source maps on $\XX$ are continuous and
  open, as is the quotient map $\XX\to \XX/H$.  Thus we have $s(x\cdot
  H)=r(y\cdot H)$.  If necessary, replace $x$ within its orbit $x\cdot
  H$ so that $s(x)=r(y)$.  However, we will \emph{not} make
  corresponding choices at this time for the $x_i$'s; rather, this
  will arise from other considerations in our argument.

  Since the quotient map $\XX\to\XX/H$ is open, after passing to
  subnets and relabeling we may, if necessary, replace each $x_i$ by
  another element of its orbit $x_i\cdot H$ so that we have $ x_i\to
  x\righttext{in}\XX$, and similarly we can assume that $y_i\to y$.
  For each $i$ there is a unique $h_i\in H$ such that $ s(x_i)\cdot
  h_i=r(y_i)$.  Then
  \begin{equation*}
    s(x_i)\cdot h_i\to r(y)=s(x),
  \end{equation*}
  and of course by continuity we also have $ s(x_i)\to s(x)$.  By the
  elementary \lemref{elementary} below we conclude that $h_i\to e$ in
  $H$. Thus, replacing $x_i$ by $x_i\cdot h_i$, we now have
  $
    s(x_i)=r(y_i)$ for all $i $,
  and we still have $x_i\to x$, and hence
  \begin{align*}
    (x_i\cdot H)(y_i\cdot H) =(x_iy_i)\cdot H \to (xy)\cdot H
    =(x\cdot H)(y\cdot H).
  \end{align*}

  Continuity of inversion is now easy: if $x_i\cdot H\to x\cdot H$ in
  $\XX/H$, then as above we can pass to a subnet (which suffices, as
  before) so that $x_i\to x$, and then
  \begin{equation*}
    (x_i\cdot H)\inv=x_i\inv\cdot H\to x\inv\cdot H=(x\cdot H)\inv.
  \end{equation*}

  Finally, assume that the action of $H$ on $\XX$ is invariant; since
  $H$ acts on the right instead of the left here, we note explicitly
  that invariance in this case means
  \begin{equation}
    \label{invariant}
    \int_\XX f(x\cdot h)\,d\lambda^u(x)
    =\int_\XX f(x)\,d\lambda^{u\cdot h}(x).
  \end{equation}
  To show that \eqref{Haar} gives a left Haar system, note that the
  range map $r:\XX\to \XX^{(0)}$ is a continuous and open surjection
  that is $H$-equivariant, the left Haar system $\lambda$ on $\XX$ is
  an $r$-system, and by \eqref{invariant} it is $H$-equivariant in
  Renault's sense \cite{ren:representation}.  Thus the existence of a
  Haar system on $\XX/H$ is guaranteed by
  \cite[Lemme~1.3]{ren:representation}.
\end{proof}

In the above proof we appealed to the following general lemma:

\begin{lem}
  \label{elementary}
  Let a locally compact group $H$ act freely and properly on the right
  of a locally compact Hausdorff space $\Omega$.  Suppose that we have
  $u\in \Omega$ and nets $\{u_i\}$ in $\Omega$ and $\{h_i\}$ in $H$
  such that $
    u_i\to u$ and $u_i\cdot h_i\to u$.
  Then $h_i\to e$.
\end{lem}

\begin{proof}
  It suffices to show that some subnet of $\{h_i\}$ converges to $e$.
  Choose a compact neighborhood $K$ of $u$.  After passing to subnets
  and relabeling, without loss of generality we may assume that
  \begin{equation*}
    u_i\in K\and u_i\cdot h_i\in K\all i.
  \end{equation*}
  Since $H$ acts properly, the set
  \begin{equation*}
    \{(v,h)\in \Omega\times H:v\in K\text{ and }v\cdot h\in K\}
  \end{equation*}
  is compact, and hence the set
  \begin{equation*}
    \{h\in H:v\cdot h\in K\text{ for some }v\in K\}
  \end{equation*}
  is also compact. Thus all the $h_i$'s lie in a compact set, so again
  by passing to subnets and relabeling we may assume that the net
  $\{h_i\}$ converges, say $h_i\to h$ in $H$. Then we have
  \begin{align*}
    u\cdot h =\lim u_i\cdot h_i =u,
  \end{align*}
  so we must have $h=e$, again by freeness.
\end{proof}

\begin{cor}\label{orbit bundle}
  Let $p:\AA\to\XX$ be a Fell bundle over a locally compact Hausdorff
  groupoid, and let $H$ be a locally compact group.  Suppose $H$ acts
  freely and properly by automorphisms on the right of $\AA$ \(by
  which we mean that the associated action of $H$ on the groupoid
  $\XX$ has these properties\).  Then the orbit space $\AA/H$ becomes
  a Fell bundle over the orbit groupoid $\XX/H$, called an \emph{orbit
    Fell bundle} or a \emph{quotient Fell bundle}, with operations
  \begin{compactenum}
  \item $p(a\cdot H)=p(a)\cdot H$;

  \item $(a\cdot H)(b\cdot H)=(ab)\cdot H$ if $s(a)=r(b)$;

  \item $(a\cdot H)^*=a^*\cdot H$.
  \end{compactenum}
\end{cor}

\begin{proof}
  It might be useful to explicitly record what a typical fibre of
  $\AA/H$ is: if $x\in\XX$, so that $x\cdot H$ is a typical element of
  $\XX/H$, then the associated fibre of $\AA/H$ is
  \begin{equation*}
    p\inv(x\cdot H)=p\inv(x)\cdot H.
  \end{equation*}
  It might also be helpful to explicitly record the norm in $\AA/H$:
  \begin{equation*}
    \|a\cdot H\|=\|a\|,
  \end{equation*}
  which is well-defined since $\|a\cdot h\|=\|a\|$ for all $h\in H$.
  The map $p$ defined in (i) is continuous and open because both the
  bundle projection $p:\AA\to\XX$ and the quotient map $\XX\to\XX/H$
  are.  Most of the axioms of a Fell bundle are routine to check.  The
  crucial property is freeness of the $H$-action; for example,
  freeness is the reason that the sum of two elements in the same
  fibre of $\AA/H$ is well-defined.

  The only properties that are not quite routine are continuity of
  multiplication and of involution.  But as in similar situations
  above, continuity can be established by lifting convergent nets from
  $\AA/H$ to $\AA$.
\end{proof}

\subsection{Principal bundles}

We now show that Fell bundles with free and proper group actions by
automorphisms are
really just special kinds of transformation bundles.  First we must
establish the analogous result for groupoids.

To motivate what follows, consider a locally compact groupoid $\YY$
acting on the left of a locally compact Hausdorff space $\Omega$, and
let $q:\Omega\to \YY^{(0)}$ be the associated fibre map, so that for
each $y\in\YY$ the map $u\mapsto y\cdot u$ is a bijection of
$q\inv(s(y))$ to $q\inv(r(y))$.  We can form the transformation
groupoid
  $\YY*\Omega=\{(y,u)\in \YY\times \Omega:s(y)=q(u)\}$,
with multiplication as in \S~\ref{sec:transf-group-bundl}.  Note that
\begin{equation*}
  (\YY*\Omega)^{(0)}=\{\YY^{(0)}*\Omega=\{(u,u)\in
   \YY^{(0)}\times \Omega:u=q(u)\}\cong \Omega
\end{equation*}

Now suppose that the above map $q:\Omega\to\YY^{(0)}$ is a principal
$H$-bundle\footnote{By which we mean that $H$ acts freely and properly
  on $\Omega$ and that $q(u)=q(v)$ if and only if $u\cdot H=v\cdot H$
  --- note that we do \emph{not} assume the bundle is locally
  trivial!}  for some locally compact group $H$, acting on the right
of $\Omega$, and assume that the actions of $H$ and $\YY$ on $\Omega$
commute.

\begin{lem}
  \label{act transformation}
  With the above set-up, $H$ acts freely and properly by automorphisms
  on the
  transformation groupoid $\YY*\Omega$ by
  \begin{equation*}
    (y,u)\cdot h=(y,u\cdot h).
  \end{equation*}
  In particular, $\pi_1:\YY*\Omega\to \YY$ is a principal $H$-bundle.
\end{lem}

\begin{proof}
  We must show that for all $h,k\in H$,
  \begin{compactenum}
  \item $(y,u)\mapsto (y,u\cdot h)$ is an automorphism of
    $\YY*\Omega$,

  \item $((y,u)\cdot h)\cdot k=(y,u)\cdot (hk)$, and

  \item $(y,u)\cdot e=(y,u)$.
  \end{compactenum}
  Since (ii) and (iii) are obvious, we only prove (i).  Let
  $s(y,u)=r(z,v)$ in $\YY*\Omega$, so that $s(y)=r(z)$ and $u=z\cdot
  v$.  Then for $h\in H$ we have
  \begin{equation*}
    u\cdot h=(z\cdot v)\cdot h=z\cdot (v\cdot h),
  \end{equation*}
  so $s((y,u)\cdot h)=r((z,v)\cdot h)$.  Thus we can multiply:
  \begin{align*}
    \bigl((y,u)\cdot h\bigr)\bigl((z,v)\cdot h\bigr) &=(y,u\cdot
    h)(z,v\cdot h) \\&=(yz,v\cdot h) \\&=(yz,v)\cdot h)
    \\&=\bigl((y,u)(z,v)\bigr)\cdot h.
  \end{align*}

  Finally, it is routine to verify that freeness and properness of the
  action of $H$ on $\Omega$ transfers to the action on $\YY*\Omega$,
  and that $\pi_1:\YY*\Omega\to \YY$ is the associated principal
  $H$-bundle.
\end{proof}

Observe in the set-up of \lemref{act transformation}, we have a commutative
diagram
\begin{equation*}
  \xymatrix{ \YY*\Omega \ar[r]^{\pi_2} \ar[d]_{\pi_1} &\Omega \ar[d]^q
    \\
    \YY \ar[r]_s &\YY^{(0)}, }
\end{equation*}
where the vertical maps are principal $H$-bundles, and that $\pi_1$ is
a surjective groupoid homomorphism, and that the map
\begin{equation*}
  (y,u)\cdot H\mapsto y
\end{equation*}
is an $H$-equivariant isomorphism of the quotient groupoid
$(\YY*\Omega)/H$ onto $\YY$ making the diagram
\begin{equation*}
  \xymatrix{ \YY*\Omega \ar[rr]^-{\pi_2} \ar[dd]_{\pi_1} \ar[dr]^{Q}
    &&\Omega \ar[dd]^q
    \\
    &(\YY*\Omega)/H \ar[dl]^\cong \ar[dr]^s
    \\
    \YY \ar[rr]_-s &&\YY^{(0)} }
\end{equation*}
commute, where $Q$ is the quotient map.

\lemref{act transformation} can be regarded as a source of free and
proper actions by automorphisms of groups on groupoids.  We now show
that, remarkably, every such action arises in this manner.  The
underlying reason is a structure theorem for principal bundles,
perhaps due to Palais (see the discussion preceding
\cite[Proposition~1.3.4]{palais}), although we have found it more
convenient to quote the version recorded in
\cite[Proposition~1.3.4]{husemoller}; we feel that this structure
theorem should be better known.

\begin{thm}
  \label{principal groupoid}
  Let $H$ be a locally compact group acting freely and properly by
  automorphisms on the right of a locally compact Hausdorff groupoid
  $\XX$, and let $\YY=\XX/H$ be the quotient groupoid. Then there is
  an action of $\YY$ on $\XX^{(0)}$ such that $\XX$ is
  $H$-equivariantly isomorphic to the transformation groupoid
  $\YY*\XX^{(0)}$.
\end{thm}

\begin{proof}
  We have a commutative diagram
  \begin{equation*}
    \xymatrix{ \XX^{(0)} \ar[d]_q &\XX \ar[l]_-{\tilde r}
      \ar[r]^-{\tilde s} \ar[d]^q &\XX^{(0)} \ar[d]^q
      \\
      \YY^{(0)} &\YY \ar[l]^-r \ar[r]_-s &\YY^{(0)} }
  \end{equation*}
  (where we denote the range and source maps for the groupoid $\XX$ by
  $\tilde r$ and $\tilde s$, respectively), in which the vertical maps
  are principal $H$-bundles and the horizontal maps are continuous.
  We have a pull-back diagram
  \begin{equation*}
    \xymatrix{ s^*(\XX^{(0)}) \ar[r]^-{\pi_2} \ar[d]_{\pi_1} &\XX^{(0)}
      \ar[d]^q
      \\
      \YY \ar[r]_-s &\YY^{(0)}, }
  \end{equation*}
  where $
    s^*(\XX^{(0)})=\{(y,u)\in \YY\times\XX^{(0)}:s(y)=q(u)\}$.
  By the universal property of pull-backs, we have a unique continuous
  map $\theta_s$ making the diagram
  \begin{equation*}
    \xymatrix{ \XX \ar[dd]_q \ar[rr]^-{\tilde s}
      \ar@{-->}[dr]^{\theta_s}_{!}  &&\XX^{(0)} \ar[dd]^q
      \\
      &s^*(\XX^{(0)}) \ar[ur]^{\pi_2} \ar[dl]^{\pi_1}
      \\
      \YY \ar[rr]_-s &&\YY^{(0)} }
  \end{equation*}
  commute, namely $
    \theta_s(x)=\bigl(q(x),\tilde s(x)\bigr)$,
  and, by \cite[Theorem~4.4.2]{husemoller}, $\theta_s$ is in fact a
  principal $H$-bundle isomorphism.  Similarly for $\theta_r:\XX\cong
  r^*(\XX^{(0)})$, and we get a big commuting diagram
  \begin{equation}
    \label{big diagram}
\begin{gathered} 
    \xymatrix{
      \XX^{(0)} \ar[dd]_q
      &&\XX \ar[dd]^q \ar[ll]_-{\tilde r} \ar[rr]^-{\tilde s}
      \ar[dl]_{\theta_r}^\cong \ar[dr]^{\theta_s}_\cong
      &&\XX^{(0)} \ar[dd]^q
      \\
      &r^*(\XX^{(0)}) \ar[ul]_{\pi_2} \ar[dr]_{\pi_1}
      &&s^*(\XX^{(0)}) \ar[ur]^{\pi_2} \ar[dl]^{\pi_1}
      \\
      \YY^{(0)}
      &&\YY \ar[ll]_-r \ar[rr]_-s
      &&\YY^{(0)}.
    }
\end{gathered}
  \end{equation}
  We use the homeomorphisms $\theta_r$ and $\theta_s$ to (temporarily)
  \emph{define} groupoid structures on the pull-backs $r^*(\XX^{(0)})$
  and $s^*(\XX^{(0)})$, and this gives rise to a groupoid isomorphism
  $\theta$ making the diagram
  \begin{equation*}
    \xymatrix{ &\XX \ar[dl]_{\theta_r} \ar[dr]^{\theta_s}
      \\
      r^*(\XX^{(0)}) &&s^*(\XX^{(0)}) \ar[ll]_-\theta^-\cong }
  \end{equation*}
  commute.  Commutativity of \eqref{big diagram} tells us that
  $\theta$ has the form $
    \theta(y,u)=(y,y\cdot u)$,
  for some map $
    (y,u)\mapsto y\cdot u$ from $s^*(\XX^{(0)}) $ to $\XX^{(0)}$.
  
We claim that this gives an action of the groupoid $\YY$ on the
  space $\XX^{(0)}$.  Continuity of the map $(y,u)\mapsto y\cdot u$ is
  clear.  The associated fibre map will be $q:\XX^{(0)}\to \YY^{(0)}$,
  which is continuous and open, and by construction the map $u\mapsto
  y\cdot u$ is a bijection of $q\inv(s(y))$ onto $q\inv(r(y))$.  To
  help us see how to show that
  \begin{equation}
    \label{associative}
    y\cdot (z\cdot u)=(yz)\cdot u
  \end{equation}
  whenever $s(y)=r(z)$ and $s(z)=q(u)$, we first observe a few
  properties of the groupoid $s^*(\XX^{(0)})$: we have
  \begin{itemize}
  \item $s(z,u)=(s(z),u)$;
  \item $r(z,u)=(r(z),z\cdot u)$;
  \item $(y,u')(z,u)$ is defined if and only if $u'=z\cdot u$, and
    then $
      (y,z\cdot u)(z,u)=(yz,u)$.
  \end{itemize}
  Now take $w\in \YY$ with $s(w)=r(y)$.  Then $s(w)=q(y\cdot (z\cdot
  u))$, so $(w,y\cdot (z\cdot u))\in s^*(\XX^{(0)})$, and
  \begin{equation*}
    s(w,y\cdot (z\cdot u)) =r(y,z\cdot u) =r\bigl((y,z\cdot
    u)(z,u)\bigr) =r(yz,u),
  \end{equation*}
  and \eqref{associative} follows.

  We now have an action of $\YY$ on $\XX^{(0)}$, and hence we can form
  the transformation groupoid $\YY*\XX^{(0)}$, which has the same
  underlying topological space as the pull-back $s^*(\XX^{(0)})$, and
  in fact the above reasoning now shows that the groupoid operation we
  temporarily gave to $s^*(\XX^{(0)})$ by insisting that the bijection
  $\theta_s:\XX\to s^*(\XX^{(0)})$ be a groupoid isomorphism,
  coincides with the operation on the transformation groupoid
  $\YY*\XX^{(0)}$.  Thus $\theta_s$ is an isomorphism of $\XX$ onto
  $\YY*\XX^{(0)}$.

  Finally, for the $H$-equivariance, we have
  \begin{align*}
    \theta_s(x)\cdot h &=(q(x),\tilde s(x))\cdot h =(q(x),\tilde
    s(x)\cdot h) =(q(x),\tilde s(x\cdot h)) \\&=(q(x\cdot h),\tilde
    s(x\cdot h)) =\theta_s(x\cdot h).  \qedhere
  \end{align*}
\end{proof}

We employ a by-now familiar technique to promote the above result from
groupoids to Fell bundles.

Let $p:\BB\to \YY$ be a Fell bundle over a locally compact Hausdorff
groupoid, let $\YY$ act on a locally compact Hausdorff space $\Omega$,
and let the map $q:\Omega\to \YY^{(0)}$ associated to the $\YY$-action
be a principal $H$-bundle for some locally compact group $H$ acting on
the right of $\Omega$.  Keep the assumptions and notation from the
preceding section, so that in particular the actions of $\YY$ and $H$
on $\Omega$ commute.

We form the transformation Fell bundle
\begin{equation*}
  \BB*\Omega=\{(b,u)\in \BB\times \Omega:(p(b),u)\in \YY*\Omega\}
\end{equation*}
over $\YY*\Omega$.

Then $H$ acts by automorphisms on $\BB*\Omega$ by
\begin{equation*}
  (b,u)\cdot h=(b,u\cdot h),
\end{equation*}
and moreover this action is free and proper in the sense that, by
\lemref{act transformation}, the associated action on $\YY*\Omega$ has
these properties.

It follows that the map
\begin{equation*}
  (b,u)\cdot H\mapsto b
\end{equation*}
is a $H$-equivariant isomorphism of the quotient Fell bundle
$(\BB*\Omega)/H$ onto $\BB$ making the diagram
\begin{equation*}
  \xymatrix{ \BB*\Omega \ar[rr]^-{p*\id} \ar[dd]_{\pi_1} \ar[dr]^{Q}
    &&\YY*\Omega \ar[dd]^{\pi_1}
    \\
    &(\BB*\Omega)/H \ar[dl]^\cong \ar[dr]^p
    \\
    \BB \ar[rr]_-p &&\YY }
\end{equation*}
commute, where $Q$ is the quotient map.

The above construction gives a source of free and proper actions (by
automorphisms) of groups on Fell bundles; again, the structure theorem
of Palais and Husemuller implies that all such actions arise in this
manner:

\begin{thm}
  \label{principal Fell}
  Let $\tilde p:\AA\to\XX$ be a Fell bundle over a locally compact
  Hausdorff groupoid, and let $H$ be a locally compact group.  Suppose
  $H$ acts freely and properly by automorphisms on the right of $\AA$,
  and let
  \begin{equation*}
    p:\BB\to \YY
  \end{equation*}
  be the quotient Fell bundle \(so that $\BB=\AA/H$ and $\YY=\XX/H$\).
  Then there is an action of $\YY$ on $\XX^{(0)}$ such that $\AA$ is
  $H$-equivariantly isomorphic to the transformation Fell bundle
  $\BB*\XX^{(0)}$.
\end{thm}

\begin{proof}
  We have a commutative diagram
  \begin{equation*}
    \xymatrix{ \AA \ar[r]^-{\tilde p} \ar[d]_q &\XX \ar[r]^-{\tilde s}
      \ar[d]_q &\XX^{(0)} \ar[d]^q
      \\
      \BB \ar[r]_-p &\YY \ar[r]_-s &\YY^{(0)}, }
  \end{equation*}
  where the vertical maps are principal $H$-bundles, and so it again
  follows from \cite[Theorem~4.4.2]{husemoller} that the map
  \begin{equation*}
    \tau(a)=\bigl(q(a),\tilde s(\tilde p(a)\bigr)
  \end{equation*}
  is a principal $H$-bundle isomorphism making the diagram
  \begin{equation*}
    \xymatrix{ \AA \ar[rr]^-{\tilde s\circ\tilde p} \ar[dd]_q
      \ar[dr]^\tau_\cong &&\XX^{(0)} \ar[dd]^q
      \\
      &\BB*\XX^{(0)} \ar[dl]^{\pi_1} \ar[ur]^{\pi_2}
      \\
      \BB \ar[rr]_-{s\circ p} &&\YY^{(0)} }
  \end{equation*}
  commute.  It is important to note that at this point the notation
  $\BB*\XX^{(0)}$ only stands for the pull-back principal $H$-bundle.
  By \thmref{principal groupoid} we have an action of $\YY$ on
  $\XX^{(0)}$ such that $\XX\cong \YY*\XX^{(0)}$.  Thus we can form
  the transformation Fell bundle $\BB*\XX^{(0)}$.

  We finish by showing that the principal $H$-bundle isomorphism
  $\tau:\AA\to \BB*\XX^{(0)}$ is a homomorphism (and hence an
  isomorphism) of Fell bundles.  If $a,b\in\AA$ with $\tilde
  s\circ\tilde p(a)=\tilde r\circ\tilde p(b)$, then
  \begin{equation*}
    \tilde s(\tilde p(a)) =\tilde r(\tilde p(b)) =q(\tilde p(b))\cdot
    \tilde s(\tilde p(b)) =p(q(b))\cdot \tilde s(\tilde p(b)),
  \end{equation*}
  so that $\tau(a)$ and $\tau(b)$ are composable in $\BB*\XX^{(0)}$,
  and we have
  \begin{align*}
    \tau(a)\tau(b) &=\bigl(q(a),\tilde s\circ\tilde
    p(a)\bigr)\bigl(q(b),\tilde s\circ\tilde p(b)\bigr)
    =(q(a)q(b),\tilde s\circ\tilde p(b)) \\&=(q(ab),\tilde
    s\circ\tilde p(ab)) =\tau(ab).
  \end{align*}
  For the involution, we have
  \begin{align*}
    \tau(a)^* &=\bigl(q(a)^*,p(q(a))\cdot \tilde s(\tilde p(a))\bigr)
    =(q(a^*),\tilde r(\tilde p(a))) \\&=(q(a^*),\tilde s(\tilde
    p(a)\inv)) =(q(a^*),\tilde s(\tilde p(a^*))) =\tau(a^*).
    \qedhere
  \end{align*}
\end{proof}

\subsection{Orbit action}

\begin{lem}
  \label{quotient action}
  Let $\XX$ be a locally compact groupoid and $H$ a locally compact
  group.  Suppose $H$ acts freely and properly by automorphisms on the
  right of $\XX$.  Then the orbit groupoid $\XX/H$ acts
  on the left of the space $\XX$ by
  \begin{equation}
    \label{quotient act}
    (x\cdot H)\cdot y=xy
    \righttext{whenever}s(x)=r(y).
  \end{equation}
\end{lem}



\begin{rem*}
  Note that $\XX$ is being used in two different ways in the statement
  of the above lemma: first as a groupoid, and second as just a space.
\end{rem*}

\begin{proof}
  This will be easier if we replace $\XX$ by the isomorphic
  transformation groupoid $\XX/H*\XX^{(0)}$, using \thmref{principal
    groupoid}.  Then the formula~\eqref{quotient act} becomes
  \begin{equation*}
    z\cdot (w,u)=(zw,u)
  \end{equation*}
  for $z,w\in \XX/H$ with $s(z)=r(w)$ and $s(w)\cdot H=u\cdot H$.  It
  is now clear that the action is well-defined and continuous, and
  that the associated map
  \begin{equation*}
    \rho:\XX/H*\XX^{(0)}\to (\XX/H)^{(0)}=\XX^{(0)}/H
  \end{equation*}
  given by
  \begin{equation*}
    \rho(x\cdot H,u)=r(x)\cdot H
  \end{equation*}
  is continuous and open (because both the range map
  $r:\XX\to\XX^{(0)}$ and the quotient map $\XX\to\XX/H$ are).
\end{proof}

\begin{cor}
  \label{orbit bundle action}
  Let $p:\AA\to\XX$ be a Fell bundle over a locally compact groupoid,
  and let $H$ be a locally compact group.  Suppose $H$ acts freely and
  properly on the right of $\AA$ by Fell-bundle automorphisms.  Then
  the orbit Fell bundle $\AA/H$ acts on the Banach bundle $\AA$ by
  \begin{equation}
    \label{orbit bundle act}
    (a\cdot H)\cdot b=ab
    \righttext{whenever}s(a)=r(b).
  \end{equation}
\end{cor}

\begin{rem*}
  Again, note that $\AA$ is being used in two different ways in the
  statement of the above corollary: first as a Fell bundle, and second
  as just a Banach bundle.
\end{rem*}

\begin{proof}
  As in \lemref{quotient action}, it is easier if we work with the
  isomorphic transformation Fell bundle $\AA/H*\XX^{(0)}$, using
  \thmref{principal Fell}.  Then the formula~\eqref{orbit bundle act}
  becomes
  \begin{equation*}
    c\cdot (d,u)=(cd,u)
  \end{equation*}
  for $c,d\in\AA$ with $s(c)=r(d)$ and $s(d)\cdot H=u\cdot H$.  Again
  it is now clear that this action is well-defined and continuous, and
  is compatible with the action of the orbit groupoid $\XX/H$ on $\XX$
  from \lemref{quotient action}.
\end{proof}

\subsection{Semidirect product orbit action}

\begin{prop}
  \label{semidirect orbit action}
  Suppose that we are given commuting free and proper actions of
  locally compact groups $G$ and $H$ on the left and right,
  respectively, by automorphisms on a locally compact groupoid $\XX$.
  Then:
  \begin{enumerate}
  \item $G$ acts on the left of the orbit groupoid $\XX/H$ by
    \begin{equation*}
      s\cdot (x\cdot H)=(s\cdot x)\cdot H;
    \end{equation*}

  \item the semidirect-product groupoid $\XX/H\rtimes G$ acts freely
    and properly on the left of the space $\XX$ by
    \begin{equation*}
      (x\cdot H,t)\cdot y=x(t\cdot y) \righttext{whenever}s(x)=r(t\cdot
      y).
    \end{equation*}
  \end{enumerate}
\end{prop}

\begin{proof}
  By \lemref{quotient action}, the orbit groupoid $\XX/H$ acts on the
  left of the space $\XX$.  It is routine to check (i), and also that
  the actions of $G$ and $\XX/H$ on $\XX$ are covariant.  Then it
  follows from \lemref{semidirect action} that $\XX/H\rtimes G$ acts
  on $\XX$ as indicated.

  To verify that the action of $\XX/H\rtimes G$ is free and proper, it
  is easier if we replace $\XX$ by the homeomorphic space
  $\XX/H*\XX^{(0)}$, using \thmref{principal groupoid}.  Then the
  action becomes
  \begin{equation*}
    (z,t)\cdot (w,u)=\bigl(z(t\cdot w),t\cdot u)
    \quad\text{if}\quad s(z)=r(t\cdot w).
  \end{equation*}
  For the freeness, if $(z,t)\cdot (w,u)=(w,u)$, then $t\cdot u=u$, so
  $t=e$ since $G$ acts freely, and then $zw=w$, so $z=r(w)$.  Thus
  $(z,t)\in (\XX/H\rtimes G)^{(0)}$.

  For the properness, if $K\subset \XX$ is compact, we can find
  compact sets
  \begin{equation*}
    K_1\subset \XX/H\midtext{and}K_2\subset \XX^{(0)}
  \end{equation*}
  such that $K\subset K_1\times K_2$.  It suffices to find a compact
  set containing any $(z,t)\in \XX/H\rtimes G$ for which $(z,t)\cdot
  K\cap K\ne \emptyset$.  If $(w,u)\in K$ and $(z,t)\cdot (w,u)\in K$,
  then $t\cdot u\in K_2$, so because $G$ acts properly there is a
  compact set $L\subset G$ containing any such $t$, and then using
  $z(t\cdot w)\in K_1$ we get
  \begin{equation*}
    z\in K_1 (t\cdot w\inv) \subset K_1(L\cdot K_1\inv),
  \end{equation*}
  which is compact in $\XX/H$.  Therefore we conclude that
  \begin{equation*}
    (z,t)\in K_1(L\cdot K_1\inv)\times L,
  \end{equation*}
  which is compact in $\XX/H\rtimes G$.
\end{proof}

\begin{prop}
  \label{semidirect orbit bundle action}
  Let $p:\AA\to\XX$ be a Fell bundle over a locally compact groupoid,
  and let $G$ and $H$ be locally compact groups.  Suppose that $G$ and
  $H$ act freely and properly on the left and right, respectively, of
  $\AA$.  Then
  \begin{enumerate}
  \item $G$ acts on the left of the orbit bundle $\AA/H$ by
    \begin{equation*}
      s\cdot (a\cdot H)=(s\cdot a)\cdot H;
    \end{equation*}

  \item the semidirect-product Fell bundle $\AA/H\rtimes G$ acts on
    the left of the Banach bundle $\AA$ by
    \begin{equation}
      \label{semidirect orbit bundle act}
      (a\cdot H,t)\cdot b=a(t\cdot b)
      \righttext{whenever}s(a)=r(p(t\cdot b)).
    \end{equation}
  \end{enumerate}
\end{prop}

\begin{proof}
  By \corref{orbit bundle action}, the orbit Fell bundle $\AA/H$ acts
  on the Banach bundle $\AA$.  By \corref{semidirect orbit action},
  $G$ acts on the orbit groupoid $\XX/H$, and it is routine to check
  (i), and also that the actions of $G$ and $\AA/H$ on $\AA$ are
  covariant.  Then (ii) follows from \corref{semidirect bundle
    action}.
\end{proof}

\end{appendix}

\bibliographystyle{amsplain}

\begin{thebibliography}{10}

\bibitem{com}
F.~Combes, \emph{{Crossed products and Morita equivalence}}, Proc. London Math.
  Soc. \textbf{49} (1984), 289--306.
  
\bibitem{ech}
S.~Echterhoff and J.~Quigg, \emph{{Induced coactions of discrete groups on
$C^*$-algebras}}, Canad. J. Math. \textbf{51} (1999), 745--770.

\bibitem{fel:ajm69}J.~M.~G. Fell, \emph{{An extension of Mackey's method to
algebraic bundles over finite groups}}, Amer. J. Math. \textbf{91} (1969),
203--238.

\bibitem{fel:mams69}J.~M.~G. Fell, \emph{{An extension of {M}ackey's method to
{B}anach {$^{\ast} $}-algebraic bundles}}, Mem. Amer. Math. Soc., No.
\textbf{90}, 1969. 

\bibitem{fd2}
J.~M.~G. Fell and R.~S. Doran, \emph{Representations of {$\sp *$}-algebras,
  locally compact groups, and {B}anach {$\sp *$}-algebraic bundles. {V}ol. 2},
  Pure and Applied Mathematics, vol. 126, Academic Press Inc., Boston, MA,
  1988.

\bibitem{gre:local}
P.~Green, \emph{{The local structure of twisted covariance algebras}}, Acta
  Math. \textbf{140} (1978), 191--250.

\bibitem{Hall1959}
M.~Hall, \emph{The Theory of Groups}, The MacMillan Co., New York, N.Y., 1959.

\bibitem{HR:mansfield}
A.~an~Huef and I.~Raeburn, \emph{{Mansfield's imprimitivity theorem for
  arbitrary closed subgroups}}, Proc. Amer. Math. Soc. \textbf{132} (2004),
  1153--1162.

\bibitem{husemoller}
D.~Husemoller, \emph{Fibre bundles}, third ed., Graduate Texts in Mathematics,
  vol.~20, Springer-Verlag, New York, 1994.

\bibitem{kmqw3}
S.~Kaliszewski, P.~S. Muhly, J.~Quigg, and D.~P. Williams, \emph{Fell
  bundles and imprimitivity 
    theorems: towards a universal generalized fixed point algebra},
  preprint, 2012.

\bibitem{kmqw1}
\bysame, \emph{Coactions and {F}ell bundles}, New York J. Math. \textbf{16}
  (2010), 315--359.

\bibitem{mackey}
G.~W. Mackey, \emph{{Imprimitivity for representations of locally compact
  groups. I}}, Proc. Natl. Acad. Sci. USA \textbf{35} (1949), 537--545.

\bibitem{man}
K.~Mansfield, \emph{{Induced representations of crossed products by
  coactions}}, J. Funct. Anal. \textbf{97} (1991), 112--161.

\bibitem{muhly:fell}
P.~S. Muhly, \emph{Bundles over groupoids}, Groupoids in analysis, geometry,
  and physics ({B}oulder, {CO}, 1999), Contemp. Math., vol. 282, Amer. Math.
  Soc., Providence, RI, 2001, pp.~67--82.

\bibitem{MuhlyCBMS}
\bysame, \emph{Coordinates in operator algebra}, CBMS Conference Lecture Notes
  (Texas Christian University 1990), In continuous preparation.

\bibitem{mrw}
P.~S. Muhly, J.~N. Renault, and D.~P. Williams, \emph{{Equivalence and
  isomorphism for groupoid $C^*$-algebras}}, J. Operator Theory \textbf{17}
  (1987), 3--22.

\bibitem{mw:fell}
P.~S. Muhly and D.~P. Williams, \emph{Equivalence and disintegration theorems
  for {F}ell bundles and their ${C}^*$-algebras}, Dissertationes Mathematicae
  \textbf{456} (2008), 1--57.

\bibitem{palais}
R.~S. Palais, \emph{On the existence of slices for actions of non-compact {L}ie
  groups}, Ann. of Math. (2) \textbf{73} (1961), 295--323.

\bibitem{rae:symmetric}
I.~Raeburn, \emph{{Induced $C^*$-algebras and a symmetric imprimitivity
  theorem}}, Math. Ann. \textbf{280} (1988), 369--387.

\bibitem{RWPullback}
I.~Raeburn and D.~P. Williams, \emph{Pull-backs of {$C^\ast$}-algebras and
  crossed products by certain diagonal actions}, Trans. Amer. Math. Soc.
  \textbf{287} (1985), no.~2, 755--777.

\bibitem{ren:representation}
J.~N. Renault, \emph{{Repr{\'e}sentation des produits crois{\'e}s
  d'alg{\`e}bres de groupoides}}, J. Operator Theory \textbf{18} (1987),
  67--97.

\bibitem{rie:induced}
M.~A. Rieffel, \emph{{Induced representations of $C^*$-algebras}}, Adv. Math.
  \textbf{13} (1974), 176--257.
  
\bibitem{rie:applications}
M.~A. Rieffel, \emph{Applications of strong Morita equivalence to transformation group $C^*$-algebras}, Operator algebras and applications, Part I (Kingston, Ont., 1980), 299--310, Proc. Sympos. Pure Math., \textbf{38}, Amer. Math. Soc., Providence, R.I., 1982.  

\bibitem{takesaki}
M.~Takesaki, \emph{Covariant representations of {$C\sp{\ast} $}-algebras and
  their locally compact automorphism groups}, Acta Math. \textbf{119} (1967),
  273--303.

\bibitem{yam:symmetric}
S.~Yamagami, \emph{{On the ideal structure of $C^*$-algebras over locally
  compact groupoids}}, preprint, 1987.

\bibitem{yam:memoir}
T.~Yamanouchi, \emph{{Duality for actions and coactions of measured groupoids
  on von Neumann algebras}}, Mem. Amer. Math. Soc. \textbf{101} (1993),
  no.~484.

\end{thebibliography}

\providecommand{\bysame}{\leavevmode\hbox to3em{\hrulefill}\thinspace}
\providecommand{\MR}{\relax\ifhmode\unskip\space\fi MR }
\providecommand{\MRhref}[2]{%
  \href{http://www.ams.org/mathscinet-getitem?mr=#1}{#2}
}
\providecommand{\href}[2]{#2}

\end{document}